\documentclass{amsart}
\usepackage{amssymb, amsmath,amsthm }

\usepackage{hyperref}
\usepackage{stmaryrd}
\usepackage[all]{xy}
\usepackage{dsfont}

\RequirePackage{mathrsfs}
\let\mathcal\mathscr

\theoremstyle{plain}

\newtheorem{prop}{Proposition}[section]

\newtheorem{lem}[prop]{Lemma}
\newtheorem{thm}[prop]{Theorem}

\newtheorem{cor}[prop]{Corollary}

\theoremstyle{remark}
\newtheorem{remar}[prop]{Remark}
\theoremstyle{definition}
\newtheorem{defi}[prop]{Definition}

\newtheorem{remark}[prop]{Remark}

\let\cal\mathcal

\def\G{{\cal G}}

\def\matrice#1#2#3#4{{\bigl(\begin{smallmatrix}#1&#2\\ #3&#4\end{smallmatrix}\bigr)}}

\DeclareMathOperator{\sgn}{sgn}

\DeclareMathAlphabet{\mathpzc}{OT1}{pzc}{m}{it}

\DeclareMathOperator{\End}{End}
\DeclareMathOperator{\Hom}{Hom}

\DeclareMathOperator{\Ind}{Ind}
\DeclareMathOperator{\cInd}{c-Ind}

\DeclareMathOperator{\GL}{GL}

\DeclareMathOperator{\Ker}{Ker}
\DeclareMathOperator{\Coker}{Coker}

\DeclareMathOperator{\Gal}{Gal}

\DeclareMathOperator{\gr}{gr}

\DeclareMathOperator{\Mod}{Mod}
\DeclareMathOperator{\Alg}{Alg}

\DeclareMathOperator{\Sp}{Sp}

\DeclareMathOperator{\Ext}{Ext}

\DeclareMathOperator{\Ban}{Ban}

\DeclareMathOperator{\dualcat}{\mathfrak C}

\DeclareMathOperator{\p1}{\bold{P}^1}
\newcommand{\cIndu}[3]{\cInd_{#1}^{#2}{#3}}

\newcommand{\Indu}[3]{\Ind_{#1}^{#2}{#3}}

\newcommand{\Q}{\mathbb{Q}}
\newcommand{\Qp}{\mathbb {Q}_p}

\newcommand{\qp}{\mathbb{Q}_p}            
\newcommand{\zp}{\mathbb{Z}_p}

\newcommand{\zpet}{\mathbb{Z}_p^{\times}}
\newcommand{\qpet}{\mathbb{Q}_p^{\times}}

\newcommand{\Zp}{\mathbb{Z}_p}

\newcommand{\Eins}{\mathbf 1}

\newcommand{\ZZ}{\mathbb Z}
\newcommand{\VV}{\mathbf V}

\newcommand{\mm}{\mathfrak m}

\newcommand{\OO}{\mathcal O}

\newcommand{\gal}{\mathcal G_{\Qp}}

\newcommand{\md}{\mathrm m}

\newcommand{\br}[1]{\llbracket #1\rrbracket}

\newcommand{\sm}{\mathrm{sm}}

\newcommand{\pro}{\mathrm{pro}}

\newcommand{\adm}{\mathrm{adm}}

\newcommand{\ladm}{\mathrm{l.adm}}

\newcommand{\wB}{\Pi(P)}

\def\O{{\cal O}}
\def\oe{\O_{\cal E}}

\title[The $p$-adic local Langlands correspondence]{The $p$-adic local Langlands correspondence\\ for ${\rm GL}_2(\qp)$}
\author{Pierre Colmez}
\address{C.N.R.S., Institut de math\'ematiques de Jussieu, 4 place Jussieu,
75005 Paris, France}
\email{colmez@math.jussieu.fr}
\author{Gabriel Dospinescu} 
\address{UMPA, \'Ecole Normale Sup\'erieure de Lyon, 46 all\'ee d'Italie, 69007 Lyon, France}
\email{gabriel.dospinescu@ens-lyon.fr}

\author{Vytautas Pa\v{s}k\={u}nas}
\address{Fakult\"{a}t f\"{u}r Mathematik, Universit\"{a}t Duisburg--Essen, 45117 Essen, Germany}
\email{paskunas@uni-due.de}
\thanks{V.P. is partially supported by the DFG, SFB/TR 45. }
\date{\today.}

\begin{document} 
\maketitle

\begin{abstract}
The $p$-adic local Langlands correspondence for ${\rm GL}_2(\qp)$
is given by an exact functor from unitary Banach representations
of ${\rm GL}_2(\qp)$ to representations of the absolute Galois group $G_{\qp}$
of $\qp$.  We prove, using characteristic $0$ methods, that this
correspondence induces a bijection between absolutely
irreducible non-ordinary representations of ${\rm GL}_2(\qp)$ and
absolutely irreducible $2$-dimensional representations of
$G_{\qp}$.  This had already been proved, by characteristic $p$
methods, but only for $p\geq 5$.
\end{abstract}

\section{Introduction}

\subsection{The $p$-adic local Langlands correspondence}

   Let $p$ be a prime number and let $G={\rm GL}_2(\qp)$. Let $L$ be a finite extension of $\qp$, with ring of integers $\mathcal{O}$, residue field 
  $k$ and uniformizer $\varpi$. 

Let ${\rm Ban}^{\rm adm}_G(L)$ be the category of admissible unitary $L$-Banach representations of $G$. 
  Any $\Pi\in {\rm Ban}^{\rm adm}_G(L)$ has an open, bounded and $G$-invariant lattice $\Theta$ and 
  $\Theta\otimes_{\mathcal{O}} k$ is an admissible smooth $k$-representation of $G$. We say that $\Pi$ in ${\rm Ban}^{\rm adm}_G(L)$ is \textit{residually of 
  finite length} if for any (equivalently, one) such lattice $\Theta$, the $G$-representation $\Theta\otimes_{\mathcal{O}} k$ is of finite length. In this case the semi-simplification of $\Theta\otimes_{\mathcal{O}} k$  is  independent of the choice of~$\Theta$, and we denote it by $\overline{\Pi}^{\rm ss}$. 
   We say that an absolutely irreducible\footnote{This means that $\Pi\otimes_{L} L'$ 
is topologically irreducible for all finite extensions $L'$ of $L$.}  
   $\Pi\in{\rm Ban}^{\rm adm}_G(L)$ is \textit{ordinary} if it is a subquotient 
of a unitary parabolic induction of a unitary character. 

    Let ${\rm Rep}_L(G)$ be the full subcategory of ${\rm Ban}^{\rm adm}_G(L)$ consisting of 
   representations $\Pi$ having a central character and which are      
   residually of finite length. 
  Let ${\rm Rep}_L(\cal{G}_{\qp})$ be the category of finite dimensional continuous $L$-representations of 
${\cal G}_{\qp}={\rm Gal}(\overline{\qp}/\qp)$.
In \cite[ch. IV]{Cbigone} is constructed an exact, covariant functor (that some people call the Montreal functor) 
$\Pi\mapsto{\mathbf V}(\Pi)$ from
${\rm Rep}_L(G)$ to ${\rm Rep}_L(\cal G_{\qp})$.   
We prove that this functor has all the properties needed to be called
the $p$-adic local Langlands correspondence for~$G$.
\begin{thm}\label{LLp}
The functor $\Pi\mapsto {\bf V}(\Pi)$ induces a bijection between  the isomorphism classes of :

$\bullet$ absolutely irreducible non-ordinary 
$\Pi\in {\rm Ban}^{\rm adm}_G(L)$,

$\bullet$ $2$-dimensional absolutely irreducible continuous 
$L$-repre\-sentations of $\cal G_{\qp}$. 
\end{thm}
Implicit in the statement of the theorem is the fact that absolutely irreducible 
$\Pi\in{\rm Ban}^{\rm adm}_G(L)$ are residually of finite length so that one can
apply the functor ${\bf V}$ to them.  

One corollary of the theorem, and of the explicit construction of the representation $\Pi(V)$ of $G$
corresponding to a representation $V$ of $\G_{\qp}$ (see below), 
is the compatibility between the $p$-adic local Langlands correspondence
and local class field theory:
we let $\varepsilon:\mathcal{G}_{\qp}\to \zpet$ be the cyclotomic character 
and we view unitary characters of $\qpet$ as characters of $\mathcal{G}_{\qp}$ via class field theory\footnote{Normalized so that 
uniformizers correspond to geometric Frobenii.} (for example, $\varepsilon$ corresponds to $x\mapsto x|x|$). 
Note that, by Schur's lemma~\cite{DS}, any absolutely irreducible object of ${\rm Ban}^{\rm adm}_G(L)$
admits a central character.
  \begin{cor} 
If $\Pi$ is an absolutely irreducible non-ordinary object of ${\rm Ban}^{\rm adm}_G(L)$
with central character $\delta$, then ${\bf V}(\Pi)$ has determinant $\delta\varepsilon$.
   \end{cor}

The next result shows that the $p$-adic
local Langlands correspondence is a refinement of the classical one (that such a statement could
be true was Breuil's starting point for his investigations on the existence of a $p$-adic
local Langlands correspondence~\cite{Brintro}).

Let $\pi$ be an admissible, absolutely irreducible, infinite dimensional,
smooth $L$-representation of $G$, and let $W$ be an algebraic representation
of $G$ (so there exist $a\in\ZZ$ and $k\geq 1$ such that $W={\rm Sym}^{k-1} L^2\otimes\det^a$).
Let $\Delta$ be the Weil representation corresponding to $\pi$ via
the classical local Langlands correspondence; we view $\Delta$ as a
$(\varphi,\G_{\qp})$-module\footnote{In general, this
can require to extend scalars to a finite unramified extension of $L$, but we assume
that this is already possible over $L$.}~\cite{Bu88l,BS}.
Let ${\cal F}(\Delta, W)$
be the space of isomorphism classes of weakly admissible, absolutely
irreducible, filtered $(\varphi,N,\G_{\qp})$-modules~\cite[Chap.~4]{Bu88sst}
whose underlying $(\varphi,\G_{\qp})$-module is isomorphic to $\Delta$ and the
jumps of the filtration are
 $-a$ and $-a-k$: if ${\cal L}\in {\cal F}(\Delta, W)$,
the corresponding~\cite{CF} representation $V_{\cal L}$ of $\G_{\qp}$
is absolutely irreducible and 
its Hodge-Tate weights are $a$ and $a+k$.
If ${\cal F}(\Delta, W)$ is not empty, it is either a point if $\pi$ is principal series
or $\p1(L)$ if $\pi$ is supercuspidal 
 or a twist of the Steinberg representation\footnote{
In this last case, the filtration corresponding to $\infty\in\p1(L)$ makes the monodromy
operator $N$ on $\Delta$ vanish and $V_{\infty}$ is crystalline (up to twist by a character) whereas,
if ${\cal L}\neq\infty$,
$V_{\cal L}$ is semi-stable non-crystalline (up to twist by a character).}.

\begin{thm}
{\rm (i)} If $\Pi$ is an admissible, absolutely
irreducible, non-ordinary, unitary completion of $\pi\otimes W$, 
then ${\bf V}(\Pi)$ is potentially
semi-stable with Hodge-Tate weights $a$ and $a+k$ and 
the underlying $(\varphi,\G_{\qp})$-module of $D_{\rm pst}({\bf V}(\Pi))$
is isomorphic to $\Delta$.

{\rm (ii)}
The functor $\Pi\mapsto D_{\rm pst}({\mathbf V}(\Pi))$ induces a 
bijection between the admissible, absolutely
irreducible, non-ordinary, unitary completions of $\pi\otimes W$ and 
${\cal F}(\Delta, W)$.
\end{thm}
The theorem follows from the combination of theorem~\ref{LLp}, \cite[th.~0.20]{Cbigone} (or~\cite{Annalen}),
\cite[th.~VI.6.42]{Cbigone} and Emerton's local-global compatibility\footnote{
It is a little bit frustrating to have to use global considerations to prove it.  By purely local considerations,
one could prove it when $\pi$ is a principal series or a twist of the Steinberg representation.  When $\pi$
is supercuspidal, one could show that there is a set $S_\pi$ of 
$(\varphi,\G_{\qp})$-modules $\Delta$ such that
the functor $\Pi\mapsto D_{\rm pst}({\mathbf V}(\Pi))$ induces a bijection between the admissible, absolutely
irreducible, unitary completions of $\pi\otimes W$ and 
the union of the ${\cal F}(\Delta, W)$, for $\Delta\in S_\pi$, but we would not know much about $S_\pi$ except
for the fact that $S_\pi\cap S_{\pi'}=\emptyset$ if $\pi\not\cong\pi'$.}
(\cite{emfm},  th.~3.2.22).

\medskip
If $p\geq 5$, the results are not new; they were proven in~\cite{cmf},
building upon~\cite{Cbigone,kisin},
via characteristic~$p$ methods, but these methods seemed to be very difficult to extend to the case $p=2$
(and also $p=3$ in a special case).  That we are able to prove the theorem in full generality
relies on a shift to characteristic~$0$ methods 
and an array of results which were not available
at the time~\cite{cmf} was written:

$\bullet$ The computation~\cite{blocks} 
of the blocks of the mod~$p$ representations of $G$, in the case $p=2$;
this computation also uses characteristic~$0$ methods.

$\bullet$ Schur's lemma for unitary Banach representations of $p$-adic Lie groups (\cite{DS} which uses results of
Ardakov and Wadsley~\cite{AW}.) 

$\bullet$ The computation~\cite{Cvectan,liu} of the locally analytic vectors of unitary principal series representations of $G$.

$\bullet$ The computation~\cite{Annalen} 
of the infinitesimal action of $G$ on locally analytic vectors of objects
of ${\rm Rep}_L(G)$.

There are 3 issues to tackle if one wants to establish theorem~\ref{LLp}: 
one has to prove that absolutely irreducible objects of ${\rm Ban}^{\rm adm}_G(L)$ are residually
of finite length and bound this length, and one has to prove surjectivity and injectivity.

\subsection{Residual finiteness}
    Before stating the result, let us introduce some notations. 
   Let $B$ be the (upper) Borel subgroup of $G$ and let $\omega: \qpet\to k^{\times}$ be the character
 $x\mapsto x|x|\pmod p$. If $\chi_1,\chi_2:\Qp^{\times}\to k^{\times}$
      are (not necessarily distinct) smooth characters, we let 
      $$\pi\{\chi_1,\chi_2\}=( {\rm Ind}_{B}^{G} \chi_1\otimes\chi_2\omega^{-1})_{\rm sm}^{\rm ss}\oplus
    ({\rm Ind}_{B}^{G} \chi_2\otimes\chi_1\omega^{-1})_{\rm sm}^{\rm ss}.$$ Then $\pi\{\chi_1,\chi_2\}$ is typically of length $2$, but it may be of length $3$, and even $4$ when $p=2$ or $p=3$
    (lemma~\ref{ss} gives an explicit description of $\pi\{\chi_1,\chi_2\}$). Recall that a smooth irreducible $k$-representation is called \textit{supersingular} if it is not  isomorphic to a subquotient of 
    some representation $\pi\{\chi_1,\chi_2\}$.     
   
     \begin{thm}\label{I}
     Let $\Pi$ be an absolutely irreducible
object of ${\rm Ban}^{\rm adm}_G(L)$.
     Then $\Pi$ is residually of finite length and, after possibly replacing $L$ by a quadratic unramified extension, 
       $\overline{\Pi}^{\rm ss}$ is either absolutely irreducible supersingular or a subrepresentation of some $\pi\{\chi_1,\chi_2\}$.
     \end{thm}
For $p\geq 5$, this theorem is proved in~\cite{cmf}.  The starting points of the proofs in~\cite{cmf}
and in this paper are the same: one starts from an absolutely irreducible mod~$p$ representation $\pi$
of $G$ and considers
the projective envelope $P$ of its dual (in a suitable category, cf.~\S~\ref{overview1}). 
 Then we are led to try to understand
the ring $E$ of endomorphisms of $P$, as this gives a description of the Banach representations of $G$
which have $\pi$ as a Jordan-H\"older component of their reduction mod~$p$.
After this the strategies of proof differ completely and only some formal parts of~\cite{cmf}
are used in this paper (mainly \S4 on Banach representations).

To illustrate the differences,
let us consider the simplest case where
$\pi$ is supersingular, so that $\VV(\pi)$ is an irreducible representation of $\gal$. The key point in both  approaches to theorem~\ref{I} 
is to prove  that the ring $E[1/p]$ is commutative. This is done as follows.

In~\cite{cmf}, the functor $\Pi\mapsto{\bf V}(\Pi)$ is used to show that $E$ surjects onto the
universal deformation ring of ${\bf V}(\pi)$, which is commutative. It is then shown that this map is an isomorphism   by showing that it induces an isomorphism on the graded
rings of $E$ and the universal deformation ring of $\VV(\pi)$ with respect to the  maximal ideals. To control the dimension of the graded pieces of $\gr^{\bullet} E$, 
one needs to be able to compute the dimension of $\Ext$-groups of mod $p$ representations of $G$. These computations become hard to handle for $p=3$ and very hard for $p=2$. 
Moreover, the argument uses that the universal deformation ring of ${\bf V}(\pi)$ is formally smooth, which fails if $p=2$ and in one case if $p=3$. 

In this paper we use the functor $\Pi\mapsto \md(\Pi)$, defined  in \cite[\S~4]{cmf},  from $\Ban^{\adm}_G(L)$ to the category of finitely generated $E[1/p]$-modules. 
We show that if $\Pi$ is the universal 
unitary completion of locally algebraic unramified
principal series representation of $G$, then the image of $E[1/p]\rightarrow \End_L(\md(\Pi))$ is commutative and then show\footnote{We actually end up 
proving a weaker statement, which is too technical for this introduction, see the proofs of corollary~\ref{standard} and theorem~\ref{main_res_fin}. To show the injectivity of \eqref{injective_map} one would additionally have  to show that the rings 
$E[1/p]/\mathfrak a_V$ in corollary~\ref{standard} are reduced. We do not prove this here, but will return to this question in \cite{cmfp2}.} that the map
\begin{equation}\label{injective_map}
E[1/p]\rightarrow \prod_{i} \End_L(\md(\Pi_i))
\end{equation} is injective, where the product is taken over all such representations.  The argument uses  the work of Berger--Breuil~\cite{bb}, that the $\Pi_i$
are admissible and absolutely irreducible, and we  can control their reductions modulo $p$ \cite{Becomp,CD}. The injectivity of \eqref{injective_map}  is morally
equivalent to the density of crystalline
representations in the
universal deformation ring of ${\bf V}(\pi)$,  and is 
more or less saying that ``polynomials are dense in continuous functions", an observation
that was used by Emerton~\cite{emfm} in a global context.  However, in our local situation,
$P$ is not finitely generated over $\O[[{\rm GL}_2(\zp)]]$, and things are more complicated
than what the above sketch would suggest; we refer the reader to \S\ref{overview1}
for a more detailed overview of the proof of the theorem.

\begin{remar}
The approach developed in~\cite{cmf}, when it works, gives more information than theorem~\ref{I}: one gets a complete description of finite length objects of 
${\rm Ban}^{\rm adm}_G(L)$, not only of its absolutely irreducible objects and also a complete description 
of the category of smooth locally admissible representations of $G$ on $\OO$-torsion modules.  However, the fact that $E$ is commutative is a very useful piece of information, and, 
 when $\pi$ is either supersingular or generic principal series,  in a forthcoming paper \cite{cmfp2} we will extend the results of \cite{cmf}  to the cases when $p=2$ and $p=3$.
 
\end{remar}

Combining theorem~\ref{I} and the fact~\cite[cor.~3.14]{DS} that an irreducible object of ${\rm Ban}^{\rm adm}_G(L)$
decomposes as the direct sum of finitely many absolutely irreducible objects after
a finite extension of $L$, we obtain the following result:
   
    \begin{cor}  An object of ${\rm Ban}^{\rm adm}_G(L)$ has finite length if and only if it
      is residually of finite length.
    \end{cor}
   The following result 
 answers question (Q3) of \cite[p.~297]{Cbigone} and is an easy consequence of theorem~\ref{I}, the exactness of the functor\footnote{More precisely, of 
 integral and torsion versions of this functor.} $\Pi\mapsto{\bf V}(\Pi)$ and ~\cite[th. 0.10]{Cbigone}.
\begin{cor}
If $\Pi\in {\rm Ban}^{\rm adm}_G(L)$ is absolutely irreducible, then $\dim_L \mathbf{V}(\Pi)\leq 2$. 
\end{cor}

\subsection{Surjectivity}\label{surjectif}
The surjectivity was proven in~\cite{Cbigone} (for $p\geq 3$, and almost for $p=2$, see below) 
by constructing, for any $2$-dimensional representation $V$ of $\G_{\Q_p}$, a representation $\Pi(V)$
of $G$ such that ${\bf V}(\Pi(V))=V$ (or $\check V$, depending on the normalisation). 
The construction goes through Fontaine's equivalence of categories~\cite{FoGrot} between
representations of $\cal G_{\qp}$ and $(\varphi,\Gamma)$-modules, as does the construction of the functor
$\Pi\mapsto {\bf V}(\Pi)$.  If $D$ is the $(\varphi,\Gamma)$-module
attached to $V$ by this equivalence of categories and if $\delta$ is a character of $\qpet$,
one can construct a $G$-equivariant sheaf $U\mapsto D\boxtimes_{\delta} U$ on $\p1=\p1(\qp)$. If  $\delta=\delta_D$, where $\delta_D=\varepsilon^{-1}\det V$, then the global sections of this sheaf
fit into  an exact sequence of $G$-representations
\begin{equation}\label{decompose_global_sections}
0\to \Pi(V)^*\otimes\delta\to D\boxtimes_\delta\p1\to \Pi(V)\to 0.
\end{equation}
The proof of the existence of this decomposition
is by analytic continuation, using explicit computations to deal with trianguline representations,
in which case $\Pi(V)$ is the universal completion of a locally analytic principal series~\cite{bb, L, PS, pem, except}, 
and the Zariski density~\cite{Ctrianguline,kisin,bockle,Chen} of trianguline (or even crystalline) representations in the
deformation space of $\overline{V}^{\rm ss}$.  
That such a strategy could work was suggested by Kisin who used a variant~\cite{kisin}
to prove surjectivity for $p\geq 5$ in a more indirect way.  This Zariski density was missing when  $p=2$ and $\overline{V}^{\rm ss}$ is scalar:
 the methods of~\cite{Ctrianguline,kisin} prove that
the Zariski closure of the trianguline (or crystalline) representations is a union of irreducible
components of the space of deformations of the residual representation; so what was really missing was
an identification of the irreducible components, which is not completely straightforward. 
This is not an issue anymore as we proved~\cite{PCDnew}
that there are exactly 2 irreducible components and that the crystalline representations are dense
in each of them.

\subsection{Injectivity}
The following result is a strengthening of the injectivity of the $p$-adic local Langlands
correspondence.
  
  \begin{thm}\label{II}
   Let $\Pi_1,\Pi_2\in {\rm Ban}^{\rm adm}_G(L)$ be absolutely irreducible, non-ordinary. 
   
   {\rm (i)} If $\mathbf{V}(\Pi_1)\cong\mathbf{V}(\Pi_2)$, then $\Pi_1\cong \Pi_2$.
   
   {\rm (ii)} We have ${\rm Hom}_{L[P]}^{\rm cont}(\Pi_1,\Pi_2)={\rm Hom}_{L[G]}^{\rm cont}(\Pi_1,\Pi_2)$, where $P$ is the mirabolic subgroup of $G$.
  \end{thm}
For absolutely irreducible non-ordinary objects of ${\rm Ban}^{\rm adm}_G(L)$, the knowledge of ${\bf V}(\Pi)$
is equivalent to that of the action of $P$ (this is not true for ordinary objects).
So, theorem~\ref{II} is equivalent to the fact that 
   we can recover an absolutely irreducible non-ordinary object $\Pi$ from its restriction to $P$. 
  If we replace $P$ with the Borel subgroup $B$, then the result follows from \cite{borel} (see also \cite[remark III.48]{CD} for a different proof). 
  The key difficulty is therefore controlling the central character, and, thanks to
results from~\cite{Cbigone,CD}, 
the proof reduces to showing that $\delta_D$ is the only character $\delta$
such that $D\boxtimes_\delta\p1$ admits a decomposition as in \eqref{decompose_global_sections}.

So assume $D\boxtimes_\delta\p1$ admits such a decomposition and set $\eta=\delta_D^{-1}\delta$.
We need to prove that $\eta=1$ and this is done in two steps: we first prove  that $\eta=1$ if it is locally constant,
and then we prove that $\eta$ is locally constant.  

The proof of step one splits into two cases:

$\bullet$ If $D$ is trianguline then we use techniques of~\cite{Cvectan,Jacquet} to study
locally analytic principal series appearing in the locally analytic vectors in $\Pi_1$ and $\Pi_2$, and make use of their universal unitary completions.

$\bullet$ 
If  $D$ is not trianguline then the restriction of global sections $D\boxtimes_\delta\p1$  to any non-empty compact open subset of $\p1$ is injective on $\Pi(V)^*\otimes\delta$,
viewed as a subspace of $D\boxtimes_\delta\p1$ via \eqref{decompose_global_sections}. 
If $\alpha\in\O^*$, let ${\cal C}^\alpha\subset D\boxtimes\zpet$ be the image of the eigenspace
of $\matrice{p}{0}{0}{1}$ for the eigenvalue $\alpha$ under the restriction to $\zpet$.
This image is the same for $\delta$ and $\delta_D$ and can be described purely in terms
of $D$ as $(1-\alpha\varphi)\cdot D^{\psi=\alpha}$.  Using the action of $\matrice{0}{1}{1}{0}$ on
 $D\boxtimes_\delta\p1$ and $D\boxtimes_{\delta_D}\p1$
we show that the ``multiplication by $\eta$'' operator
$m_\eta:D\boxtimes\zpet\to D\boxtimes\zpet$ (see ${\rm n}^{\rm o}$~\ref{analytic} for a precise definition)
sends ${\cal C}^\alpha$ into ${\cal C}^{\alpha\eta(p)}$, 
and using the above-mentionned injectivity, that $\eta=1$.

To prove that $\eta$ is locally constant, one can, in
most cases, use the formulas~\cite{Annalen} for the
infinitesimal action of $G$.  In the remaining cases one uses the fact that
the characters $\eta$ sending ${\cal C}^\alpha$ into ${\cal C}^{\alpha\eta(p)}$
for all $\alpha$ form a Zariski closed subgroup of the space of all characters, and such
a subgroup automatically contains a non-trivial locally constant character if
it is not reduced to $\{1\}$. 

The reader will find a more detailed overview of the proof in \S\ref{overviewCD}.

\medskip
Finally, we
  give a criterion for an absolutely irreducible object of 
   ${\rm Ban}^{\rm adm}_G(L)$ to be non-ordinary; this refines theorem~\ref{I}, by describing
 in which case the inclusion
   $\overline{\Pi}^{\rm ss}\subset \pi\{\chi_1,\chi_2\}$ given by this theorem is an equality.
  It is a consequence of  
theorems~\ref{I} and~\ref{II}, and of the compatibility \cite{Becomp, CD} of $p$-adic and mod $p$ Langlands correspondences.
  
  \begin{thm}\label{nonordinary}
   
   Let $\Pi\in {\rm Ban}^{\rm adm}_G(L)$ be absolutely irreducible. The following assertions are equivalent
   
   {\rm (i)}  $\mathbf{V}(\Pi)$ is $2$-dimensional. 
   
  {\rm (ii)} $\Pi$ is non-ordinary. 
   
   {\rm (iii)} After possibly replacing $L$ by a quadratic unramified extension, $\overline{\Pi}^{\rm ss}$ is either absolutely irreducible supersingular or isomorphic to some $\pi\{\chi_1,\chi_2\}$. 
    \end{thm}
    
\subsection{Acknowledgements}
V.P. would like to thank Matthew Emerton for a number of stimulating discussions. In particular, \S~\ref{PI_rings} is closely related to a joint and ongoing work with Emerton.  
  
\section{Residual finiteness}

\subsection{Overview of the proof}\label{overview1}
If $G$ is any $p$-adic analytic group, let $\Mod^{\sm}_G(\OO)$  be the category of smooth representations  of $G$ on $\mathcal{O}$-torsion modules.  
 Pontryagin duality induces an anti-equivalence of categories between $\Mod^{\sm}_G(\OO)$ and a certain category $\Mod^{\pro}_G(\OO)$ of linearly compact $\OO$-modules with a continuous $G$-action, 
see \cite{ord1}. In particular, if $G$ is compact then $\Mod^{\pro}_G(\OO)$ is the category of compact $\OO\br{G}$-modules, where 
$\OO\br{G}$ is the completed group algebra.  
 Let $\Mod^{?}_G(\OO)$ be a full subcategory of $\Mod^{\sm}_G(\OO)$ closed under subquotients and arbitrary direct sums in 
 $\Mod^{\sm}_G(\OO)$
and such that representations in $\Mod^{?}_G(\OO)$  are equal to the union of their subrepresentations of finite length. Let 
$\dualcat(\OO)$ be the full subcategory of  $\Mod^{\pro}_G(\OO)$ antiequivalent to $\Mod^{?}_G(\OO)$
via the Pontryagin duality.

Let $\pi\in {\rm Mod}_{G}^{?}(\mathcal{O})$ be admissible and absolutely irreducible, let $P\twoheadrightarrow \pi^{\vee}$ be a
projective envelope of $\pi^{\vee}$ in $\dualcat(\OO)$, and let $E:=\End_{\dualcat(\OO)}(P)$.

 Let $\Pi\in {\rm Ban}^{\rm adm}_G(L)$
and let $\Theta$ be an open, bounded and $G$-invariant lattice in $\Pi$. Let  
 $\Theta^d={\rm Hom}_{\mathcal{O}}(\Theta, \mathcal{O})$ be the Schikhof dual of $\Theta$. Endowed with the topology of pointwise convergence, $\Theta^d$ 
is an object of $\Mod^{\pro}_G(\OO)$, see \cite[Lem.4.4]{cmf}. If $\Theta^d$ is in $\dualcat(\OO)$ then $\Xi^d$ is in $\dualcat(\OO)$
for every open bounded $G$-invariant lattice $\Xi$ in $\Pi$, since $\Theta$ and $\Xi$ are commensurable and $\dualcat(\OO)$ is closed under subquotients, see \cite[Lem.4.6]{cmf}. We let $\Ban^{\adm}_{\dualcat(\OO)}$ be the full subcategory of 
$\Ban^{\adm}_G(L)$ consisting of those $\Pi$ with $\Theta^d$ in $\dualcat(\OO)$. For $\Pi\in\Ban^{\adm}_{\dualcat(\OO)}$
we let 
$$\md(\Pi):=\Hom_{\dualcat(\OO)}(P, \Theta^d)\otimes_{\OO} L.$$
Then $\md(\Pi)$ is a right $E[1/p]$-module which does not depend on the choice of $\Theta$, since
any two open, bounded lattices in $\Pi$ are commensurable.  The functor $\Pi\mapsto \md(\Pi)$ from $\Ban^{\adm}_{\dualcat(\OO)}$ to the category of right $E[1/p]$-modules is exact by \cite[Lem.4.9]{cmf}.  The proposition 
below is proved in \cite[\S~4]{cmf}, as we explain in \S\ref{main_section}.

 \begin{prop} \label{easy} For  $\Pi$ in $\Ban^{\adm}_{\dualcat(\OO)}$  the following assertions hold:
\begin{itemize}
\item[(i)] $\md(\Pi)$ is a finitely generated $E[1/p]$-module;
\item[(ii)] $\dim_L \md(\Pi)$ is equal to the multiplicity with which $\pi$ occurs as a subquotient of $\Theta\otimes_{\mathcal{O}} k$;
\item[(iii)] if $\Pi$ is topologically irreducible, then 
\begin{itemize} 
\item[a)] $\md(\Pi)$ is an irreducible $E[1/p]$-module;
\item[b)]  the natural map $\End_G^{\rm cont}(\Pi) \rightarrow \End_{E[1/p]}(\md(\Pi))^{op}$ is an isomorphism.
\end{itemize}
\end{itemize}
\end{prop}

Let us suppose that we are given a family $\{ \Pi_i\}_{i\in I}$ in ${\rm Ban}^{\rm adm}_G(L_i)$,
where for each $i\in I$, $L_i$ is a finite extension of $L$ with residue field $k_i$. Let us further suppose that each $\Pi_i$  lies in $\Ban^{\adm}_{\dualcat(\OO)}$, when considered as $L$-Banach representation.
Suppose that $d\geq 1$ is an integer such that we can find open, bounded and $G$-invariant lattices $\Theta_i$ in $\Pi_i$ such that
$\pi\otimes_k k_i$ occurs with multiplicity $\leq d$ as a subquotient of $\Theta_i\otimes_{\mathcal{O}_{L_i}} k_{i}$ for all $i\in I$. Thus 
$\pi$ occurs with multiplicity $\leq [L_i:L]d$ as a subquotient of $\Theta_i/(\varpi)$ and 
proposition~\ref{easy} yields $\dim_{L_i} \md(\Pi_i)\le d$. For simplicity let us further assume that $d=1$, then we can conclude that the action of $E[1/p]$ induces a homomorphism $E[1/p]\rightarrow \End_{L_i}(\md(\Pi_i))\cong L_i$. If $\mathfrak a_i$ is the kernel of this map then $E[1/p]/\mathfrak a_i$ is commutative, and hence if we let $\mathfrak a=\cap_{i\in I} \mathfrak a_i$, then we deduce that $E[1/p]/\mathfrak a$ is commutative.  Let us further assume that $\mathfrak a=0$. Then we can conclude that the ring $E$ is commutative. Let 
$\Pi$ in $\Ban^{\adm}_{\dualcat(\OO)}$ be absolutely irreducible, and let $\mathcal E$ be the image of $E[1/p]$ in $\End_L(\md(\Pi))$. 
Since $E[1/p]$ is commutative, so is $\mathcal E$, and using proposition~\ref{easy} (iii) b) we deduce that $\mathcal E$ is a subring 
of $\End_G^{\rm cont}(\Pi)$. 

Now comes a new ingredient, not available at the time of 
writing \cite{cmf}: by Schur's lemma \cite{DS}, 
since $\Pi$ is absolutely irreducible we have $\End_G^{\rm cont}(\Pi)=L$, hence $\mathcal E= L$. Since $\md(\Pi)$ is an irreducible $E[1/p]$-module by proposition~\ref{easy} (iii) a), we conclude 
that $\dim_L \md(\Pi)=1$, and hence by part (ii) of the proposition we conclude that $\pi$ occurs with multiplicity $1$ as subquotient of $\Theta/\varpi$. 
If $d>1$ one can still run the same argument concluding that $\pi$ occurs with multiplicity at most $d$ as subquotient of $\Theta/(\varpi)$ by using rings with polynomial identity.

  All the previous constructions and the strategy of proof explained above work in great generality ($G$ was any $p$-adic analytic group), provided certain conditions are satisfied, the hardest of which is finding a family $\{ \Pi_i\}_{i\in I}$, which enjoys all these nice properties. From now on we let $G=\GL_2(\Qp)$, and let $\Mod^?_G(\OO)$ be the category of locally admissible representations $\Mod^{\ladm}_G(\OO)$. This category, introduced by Emerton in \cite{ord1}, consists of all representations in $\Mod^{\sm}_G(\OO)$, 
which are equal to the union of their admissible subrepresentations.
For the family $\{\Pi_i\}_{i\in I}$ we take all the 
Banach representations corresponding to $2$-dimensional crystalline representations of $\gal$. 
It follows from the explicit description~\cite{Becomp} of  $\overline{\Pi}_i^{\rm ss}$, 
that $\pi$ can occur as  a subquotient  with multiplicity at most $2$, and multiplicity one if $\pi$ is either supersingular or generic principal series. 

The statement $\cap_{i\in I} \mathfrak a_i=0$
morally is the statement ``crystalline points are dense in the universal deformation ring", so one certainly 
expects it to be true, since on the Galois side this statement
is known~\cite{Ctrianguline,kisin} to be 
true\footnote{If we assume that $p\ge 5$ then using results of \cite{cmf} one may show that the assertion on the Galois side implies the assertion on the $\GL_2(\Qp)$-side.}. 
In fact, Emerton 
has proved an analogous global statement ``classical crystalline points are dense in the big Hecke algebra" by using $\GL_2$--methods, \cite[cor.~5.4.6]{emfm}.
The Banach representation  denoted by $\wB$ in \S\ref{main_section} is a local analog of Emerton's completed cohomology. However, Emerton's argument does not 
seem to carry over directly, since although  $P$ is projective in $\Mod^{\pro}_K(\OO)$, it is  not a finitely generated $\OO\br{K}$--module, 
and in our context the locally algebraic vectors in $\wB$ are not a  semi-simple representation of $\GL_2(\Qp)$. 
Because of this we do not prove directly that $\cap_{i} \mathfrak a_i=0$, but a weaker statement, which suffices for the argument to work. 
To get around the issue that $P$ is not finitely generated over $\OO\br{K}$ we have to perform some tricks, see propositions~\ref{Omega}, \ref{dense}.

\subsection{Proof of proposition~\ref{easy}}\label{main_section} 
\begin{proof}  
Part (i) is \cite[prop. 4.17]{cmf}.

Part (ii) follows from the proof of \cite[Lem. 4.15]{cmf}, which unfortunately assumes $\Theta^d\otimes_{\OO} k$ to be of finite length. This assumption is not necessary:
since $\Theta^d$ is an object of $\dualcat(\OO)$  we may write $\Theta^d\otimes_{\OO} k \cong \varprojlim M_i$, where the projective limit is taken over all the finite length quotients. Since $P$ is projective we obtain an isomorphism $\Hom_{\dualcat(\OO)}(P, \Theta^d\otimes_{\OO} k)\cong \varprojlim \Hom_{\dualcat(\OO)}(P, M_i)$. 
Since $M_i$ are of finite length, \cite[Lem. 3.3]{cmf} says that $\dim_k \Hom_{\dualcat(\OO)}(P, M_i)$ is equal to the multiplicity with which $\pi^{\vee}$ occurs in $M_i$ as a subquotient, which is the same as multiplicity with which $\pi$ occurs in $M_i^{\vee}$ as a subquotient.
Dually we obtain $\Theta\otimes_{\OO} k\cong (\Theta^d\otimes_{\OO} k)^{\vee}\cong \varinjlim M_i^{\vee}$, which allows to conclude that $\pi$ occurs with finite multiplicity in $\Theta\otimes_{\OO} k$ if and only if $\dim_k \Hom_{\dualcat(\OO)}(P, \Theta^d\otimes_{\OO} k)$ is finite, in which case both numbers coincide. 
Since $P$ is a compact flat $\mathcal{O}$-module and a projective object in $\dualcat(\OO)$, it follows that 
 $\Hom_{\dualcat(\OO)}(P, \Theta^d)$ is a compact, flat $\OO$-module, which is congruent to $\Hom_{\dualcat(\OO)}(P, \Theta^d/(\varpi))$ modulo $\varpi$, thus $\dim_L \Hom_{\dualcat(\OO)}(P, \Theta^d)\otimes_{\OO} L= \dim_k \Hom_{\dualcat(\OO)}(P, \Theta^d\otimes_{\OO} k)$.
 
 Part (iii) a) is  \cite[prop. 4.18 (ii)]{cmf} and Part (iii) b) is \cite[prop. 4.19]{cmf}.
\end{proof}

\subsection{Rings with polynomial identity}\label{PI_rings}

\begin{defi} Let $R$ be a (possibly non-commutative) ring and let $n$ be a natural number. We say that $R$ satisfies the standard identity $s_n$ if for 
every $n$-tuple $\Phi=(\phi_1, \ldots, \phi_n)$ of elements of $R$ we have $s_n(\Phi):= \sum_{\sigma} \sgn(\sigma) \phi_{\sigma(1)}\ldots \phi_{\sigma(n)}=0$, 
where the sum is taken over all the permutations of the set $\{1, \ldots, n\}$.
\end{defi}

\begin{remar}\label{commute} (i) Since $s_2(\phi_1, \phi_2)= \phi_1 \phi_2- \phi_2 \phi_1$, the ring $R$ satisfies the standard identity $s_2$ if and only if $R$ is commutative.

(ii) We note that $R$ satisfies $s_n$ if and only if the opposite ring $R^{op}$ satisfies $s_n$.

(iii) Let $\{\mathfrak{a}_i\}_{i\in I}$ be a family of ideals of $R$ such that $\bigcap_{i\in I} \mathfrak{a}_i=0$. Then 
$R$ satisfies $s_n$ if and only if $R/{\mathfrak a}_i$ satisfies $s_n$ for all $i\in I$. 

(iv) By a classical result of Amitsur and Levitzki \cite[Thm.13.3.3]{noncommute}, for any commutative ring $A$, the ring $M_n(A)$ satisfies the standard identity $s_{2n}$. 
\end{remar}

\begin{lem}\label{PI_local} Let $A$ be a commutative ring, $n\geq 1$ and let $M$ be an $A$-module which is a quotient of $A^n$. Then 
${\rm End}_A(M)$ satisfies the standard identity $s_{2n}$.
\end{lem}
\begin{proof} By hypothesis there are $e_1,...,e_n\in M$ generating $M$ as an $A$-module. Let $\phi_1,...,\phi_{2n}\in {\rm End}_A(M)$ 
and let $X^{(1)},...,X^{(2n)}\in M_n(A)$ be matrices such that $\phi_k(e_i)=\sum_{j=1}^{n} X^{(k)}_{ji} e_j$ for all $i\leq n$ and 
$k\leq 2n $. Setting $X=s_{2n}(X^{(1)},...,X^{(2n)})$, 
 for all $i\leq n$ we have 
$$s_{2n}(\phi_1,...,\phi_{2n})(e_i)=\sum_{j=1}^{n} X_{ji}e_j.$$
By remark~\ref{commute} (iv) we have $X=0$ and the result follows. 
\end{proof}

\begin{lem}\label{get_it_right} Let $A$ be a commutative noetherian ring, let $M$ be an $A$-module, such that every finitely generated submodule is of finite length, and let $n$ be an integer. If $\dim_{\kappa(\mm)} M[\mm]\le n$ for every maximal ideal $\mm$ of $A$ 
then $\End_A(M)$ satisfies the standard identity $s_{2n}$.
\end{lem}
\begin{proof} The assumption on $M$ implies that $M\cong \oplus_{\mm} M[\mm^{\infty}]$, where the sum is taken over all the maximal ideals in $A$ and
$M[\mm^{\infty}]=\varinjlim M[\mm^n]$, where $M[\mm^n]=\{ m\in M: a m =0, \forall a\in\mm^n\}$. Since $M[\mm^{\infty}]$ is only supported on $\{\mm\}$,  if $\mm_1\neq \mm_2$ then
$\Hom_{A}(M[\mm_1^{\infty}], M[\mm_2^{\infty}])=0$. Thus $\End_A(M)\cong \prod_{\mm} \End_A(M[\mm^{\infty}])$ and it is enough to show 
the assertion in the case when $M=M[\mm^{\infty}]$, which we now assume.  

Since in this case $\End_A(M)=\End_{\hat{A}}(M)$, where $\hat{A}$ is the $\mm$-adic 
completion of $A$, we may further assume that $(A, \mm)$ is a complete local ring. Let $E(\kappa(\mm))$ be an injective envelope of $\kappa(\mm)$ in the 
category of $A$-modules. The functor $(\ast)^{\vee}:=\Hom_{A}(\ast,  E(\kappa(\mm)))$ induces an anti-equivalence of categories between artinian and noetherian $A$-modules, see \cite[Thm. A.35]{twenty4}. Hence, $\End_A(M)\cong \End_A(M^{\vee})^{op}$. Since $M[\mm]\hookrightarrow M$ is essential, we may embed
$M\hookrightarrow E(\kappa(\mm))^{\oplus d}$, where $d=\dim_{\kappa(\mm)} M[\mm]$. Since $E(\kappa(\mm))^{\vee}\cong A$ by \cite[Thm. A.31]{twenty4}, we obtain a surjection $A^{\oplus d}\twoheadrightarrow M^{\vee}$. We deduce from lemma~\ref{PI_local} that $\End_A(M^{\vee})$ satisfies the standard identity $s_{2n}$.
\end{proof}

\subsection{Density}\label{density}
Let $K$ be a pro-finite group with an open pro-$p$ group. Let $\OO\br{K}$ be the completed group algebra, and 
let $\Mod^{\pro}_K(\OO)$ be the category of compact linear-topological $\OO\br{K}$-modules. Let $\{V_i\}_{i\in I}$ be a family of
continuous representations 
of $K$ on finite dimensional $L$-vector spaces, and let $M\in \Mod^{\mathrm{pro}}_K(\OO)$. 

\begin{defi}\label{capture}We say 
that $\{V_i\}_{i\in I}$ \textit{captures  $M$} if the smallest quotient
$M\twoheadrightarrow Q$, such that
 $\Hom_{\OO\br{K}}^{\rm cont}(Q, V_i^*)\cong \Hom_{\OO\br{K}}^{\rm cont}(M, V_i^*)$ for all $i\in I$ is equal to  $M$. 
\end{defi}

\begin{lem}\label{intersect} Let $N=\bigcap_{\phi} \Ker \phi$, where the intersection is taken over all $\phi\in \Hom_{\OO\br{K}}^{\rm cont}(M, V_i^*)$, for all $i\in I$.
Then $\{V_i\}_{i\in I}$ captures $M$ if and only if $N=0$.
\end{lem}
\begin{proof} It is immediate that $\Hom_{\OO\br{K}}^{\rm cont}(M/N, V_i^*)\cong \Hom_{\OO\br{K}}^{\rm cont}(M, V_i^*)$ for all 
$i\in I$. This implies the assertion.
\end{proof}

\begin{lem}\label{capture_sub} Let $M'$ be a closed $\OO\br{K}$-submodule of $M$. If $\{V_i\}_{i\in I}$ captures $M$, then it also captures $M'$.
\end{lem}
\begin{proof} Let $v\in M'$ be non-zero. Since $\{V_i\}_{i\in I}$ captures $M$, lemma~\ref{intersect} implies that  there exist 
$i\in I$ and $\phi\in \Hom^{\rm cont}_{\OO\br{K}}(M, V_i^*)$, such that $\phi(v)\neq 0$. Thus $\bigcap_{\phi} \Ker \phi=0$, where   the intersection 
is taken over all $\phi\in \Hom_{\OO\br{K}}^{\rm cont}(M', V_i^*)$, for all $i\in I$. Lemma~\ref{intersect} implies that $\{V_i\}_{i\in I}$ captures $M'$.
\end{proof}

\begin{lem}\label{faithful} Assume that $\{ V_i\}_{i\in I}$ captures $M$ and  let $\phi\in \End_{\OO\br{K}}^{\rm cont}(M)$. If $\phi$ kills 
$\Hom_{\OO\br{K}}^{\rm cont}(M, V_i^*)$ for all $i\in I$ then $\phi=0$.
\end{lem}
\begin{proof} The assumption on $\phi$ implies that
$\Hom_{\OO\br{K}}^{\rm cont}(\Coker \phi, V_i^*)\cong \Hom_{\OO\br{K}}^{\rm cont}(M, V_i^*)$ for all $i\in I$. Since $\{V_i\}_{i\in I}$ captures $M$, we 
deduce that $M=\Coker \phi$ and thus $\phi=0$.
\end{proof}

\begin{lem}\label{capture_banach} Let $M\in \Mod^{\pro}_K(\OO)$ be $\OO$-torsion free, let $\Pi(M):=\Hom_{\OO}^{\rm cont}(M, L)$ be an $L$-Banach space equipped with a supremum norm. Then 
$\{V_i\}_{i\in I}$ captures $M$ if and only if the image of the evaluation map $\oplus_i \Hom_K(V_i, \Pi(M))\otimes_L V_i \rightarrow \Pi(M)$ is a dense subspace.
\end{lem}
\begin{proof} It follows from \cite[Thm.1.2]{iw} that the evaluation map $M \times \Pi(M)\rightarrow L$ induces an isomorphism 
\begin{equation}\label{ST_iso1}
M\otimes_{\OO} L \cong \Hom^{\rm cont}_L(\Pi(M), L).
\end{equation}
If $\varphi\in \Hom_{\OO\br{K}}^{\rm cont}(M, V_i^*)$ then we define $\varphi^d\in \Hom_K(V_i, \Pi(M))$ by $\varphi^d(v)(m):= \varphi(m)(v).$ It follows from \cite[Thm.1.2]{iw} that 
the map $\varphi\mapsto \varphi^d$ induces an isomorphism 
\begin{equation}\label{ST_iso2}
\Hom_{\OO\br{K}}^{\rm cont}(M, V_i^*)\cong \Hom_K(V_i, \Pi(M)).
\end{equation}
Let $m\in M$ and let $\ell_m$ be the image of $m$ in  $\Hom^{\rm cont}_L(\Pi(M), L)$ under \eqref{ST_iso1}. Then for all $i\in I$ and all $\varphi\in \Hom_{\OO\br{K}}^{\rm cont}(M, V_i^*)$, 
$\varphi(m)=0$ if and only if $\ell_m \circ \varphi^d=0$. Using lemma~\ref{intersect} and isomorphisms \eqref{ST_iso1}, \eqref{ST_iso2} we deduce that $\{V_i\}_{i\in I}$ does not capture $M$ if and only if 
the image of the evaluation map $\oplus_i \Hom_K(V_i, \Pi(M))\otimes_L V_i \rightarrow \Pi(M)$ is not a dense subspace. 
\end{proof}

\begin{lem}\label{product} The following assertions are equivalent:
\begin{itemize} 
\item[(i)] $\{V_i\}_{i\in I}$ captures every indecomposable projective in $\Mod^{\pro}_{K}(\OO)$;
\item[(ii)]  $\{V_i\}_{i\in I}$ captures every  projective in $\Mod^{\pro}_{K}(\OO)$;
\item[(iii)] $\{V_i\}_{i\in I}$ captures $\OO\br{K}$.
\end{itemize}
\end{lem}
\begin{proof} (i) implies (ii).  Let $P$ be a projective object in $\Mod^{\pro}_{K}(\OO)$. Then $P\cong \prod_{j\in J} P_j$, where $P_j$ is projective indecomposable 
for every $j\in J$, see \cite[V.2.5.4]{demgab}. For each $j\in J$ let $p_j: P\rightarrow P_j$ denote the projection. Since $\{V_i\}_{i\in I}$ captures
$P_j$ by assumption, it follows from lemma~\ref{intersect} that $\Ker p_j= \cap_{\phi} \Ker \phi\circ p_j$, where the intersection is taken over all 
$\phi\in \Hom_{\OO\br{K}}^{\rm cont}(P_j, V_i^*)$, for all $i\in I$. Since $\cap_{j\in J} \Ker p_j =0$, we use lemma~\ref{intersect} again to deduce that 
$\{V_i\}_{i\in I}$ captures $P$. 

(ii) implies (iii), as $\OO\br{K}$ is projective in $\Mod^{\pro}_{K}(\OO)$. 

(iii) implies (i). Every indecomposable projective object in  $\Mod^{\pro}_{K}(\OO)$ is a direct summand of $\OO\br{K}$, see for example \cite[prop.~4.2]{comp}. The assertion follows from
lemma~\ref{capture_sub}.
\end{proof} 

Let $G$ be an affine group scheme of finite type over $\Zp$ such that $G_L$ is a split connected reductive group over $L$. Let 
$\Alg(G)$ be the set isomorphism classes of irreducible rational representations of $G_L$, which we view as representations of $G(\Zp)$ via the inclusion $G(\Zp)\subset G(L)$.
\begin{prop}\label{capture_algebraic} $\Alg(G)$  captures every projective object in $\Mod^{\pro}_{G(\Zp)}(\OO)$, 
\end{prop}
\begin{proof} The  proof is very much motivated  by \cite[prop. 5.4.1]{emfm}, which implies the statement for $G=\GL_2$.
Let $K=G(\Zp)$ and let $\mathcal C(K, L)$ be the space of continuous functions from $K$ to $L$. Since $K$ is compact, the supremum norm 
makes $\mathcal C(K, L)$ into a unitary  $L$-Banach representation of $K$. It is shown in \cite[Lem.2.1, cor.~2.2]{iw} 
that the natural map $K\rightarrow \OO\br{K}$, $g\mapsto g$ induces an isometrical, $K$-equivariant isomorphism between
$\mathcal C(K, L)$ and $\Hom_{\OO}^{\rm cont}(\OO\br{K}, L)$. It is shown in \cite[prop.~A.3]{blocks} that the image of the evaluation map
$\oplus \Hom_K(V, \mathcal C(K, L))\otimes V\rightarrow \mathcal C(K, L)$ is a dense subspace, where the sum is taken over all $V\in \Alg(G)$.
 Lemma~\ref{capture_banach} implies that $\Alg(V)$ captures $\OO\br{K}$, and the assertion follows from lemma~\ref{product}.
\end{proof}

\subsection{Locally algebraic vectors in $\Pi(P)$}

From now on let $G=\GL_2(\Qp)$, $K=\GL_2(\Zp)$, and let $\pi$ be an admissible smooth, absolutely irreducible $k$-representation of $G$.
Recall that if 
$\chi_1, \chi_2: \Qp^{\times}\rightarrow k^{\times}$ are smooth characters, then 
   $$\pi\{\chi_1,\chi_2\}:= ({\rm Ind}_{B}^{G} \chi_1\otimes\chi_2\omega^{-1})_{\rm sm}^{\rm ss}\oplus
    ({\rm Ind}_{B}^{G} \chi_2\otimes\chi_1\omega^{-1})_{\rm sm}^{\rm ss}.$$

\begin{defi}\label{dofpi}
 If $\pi$ is supersingular, let $d(\pi)=1$. Otherwise, there is a unique
   $\pi\{\chi_1,\chi_2\}$
 containing $\pi$ and we let $d(\pi)$ be the multiplicity of $\pi$ in $\pi\{\chi_1,\chi_2\}$. 
\end{defi}

\begin{lem}\label{ss}
 Let $\chi_1, \chi_2: \Qp^{\times}\rightarrow k^{\times}$ be smooth characters. Then $\pi\{\chi_1,\chi_2\}$ is isomorphic to one of the following:
\begin{itemize}
\item[(i)] $(\Indu{B}{G}{\chi_1\otimes\chi_2\omega^{-1}})_{\sm}\oplus (\Indu{B}{G}{\chi_2\otimes\chi_1\omega^{-1}})_{\sm}$, if $\chi_1\chi_2^{-1}\neq \Eins, \omega^{\pm 1}$;
\item[(ii)] $(\Indu{B}{G}{\chi\otimes\chi\omega^{-1}})_{\sm}^{\oplus 2}$, if $\chi_1=\chi_2=\chi$ and $p\ge 3$;
\item[(iii)] $(\Eins\oplus \Sp\oplus \Indu{B}{G}{\omega \otimes \omega^{-1}})\otimes\chi\circ\det$, if $\chi_1\chi_2^{-1}=\omega^{\pm 1}$ and $p\ge 5$;
\item[(iv)] $(\Eins \oplus \Sp\oplus \omega\circ\det\oplus \Sp\otimes\omega\circ \det)\otimes \chi\circ \det$, if $\chi_1\chi_2^{-1}=\omega^{\pm 1}$ and $p=3$;
\item[(v)] $(\Eins \oplus \Sp)^{\oplus 2}\otimes\chi\circ \det$ if $\chi_1=\chi_2$ and $p=2$.
\end{itemize}
 In particular, $d(\pi)=1$ unless we are in one of the following cases, when $d(\pi)=2$:
 \begin{itemize}
 \item[(a)] $p\ge 3$ and $\pi\cong (\Indu{B}{G}{\chi\otimes \chi\omega^{-1}})_{\sm}$  or 
\item[(b)]$p=2$ and either $\pi\cong \chi\circ \det$ or $\pi\cong \Sp\otimes\chi\circ \det$, for some smooth character $\chi: \Qp^{\times}\rightarrow k^{\times}$.
\end{itemize}
\end{lem}

\begin{proof} The representation
$(\Indu{B}{G}{\chi_1\otimes\chi_2\omega^{-1}})_{\sm}$ is irreducible if and only if $\chi_1\neq \chi_2\omega^{-1}$, otherwise its semi-simplification consists of a character and a twist of the Steinberg 
representation, see \cite[Thm. 30]{bl}. The result follows.
\end{proof}

Let $\Mod^{\ladm}_G(\OO)$ be the category of locally admissible representations introduced by Emerton in \cite{ord1}. Proposition 2.2.18 in \cite{ord1} shows that  $\Mod^{\ladm}_G(\OO)$ is closed under subquotients and arbitrary direct sums in  
 $\Mod^{\sm}_G(\OO)$, and theorem 2.3.8 in \cite{ord1} implies that every locally admissible representation is a union of its subrepresentations of finite length. So $\Mod^{\ladm}_{G}(\OO)$ satisfies the conditions imposed on $\Mod^{?}_{G}(\OO)$ in 
 \S\ref{overview1}. Let $\dualcat(\OO)$ be the full subcategory of $\Mod^{\pro}_G(\OO)$, which is anti-equivalent to 
 $\Mod^{\ladm}_{G}(\OO)$ via  Pontryagin duality.  We have $\Ban^{\adm}_{\dualcat(\OO)}=\Ban_G^{\adm}(L)$.
 
 Let $P\twoheadrightarrow \pi^{\vee}$ be a projective envelope of $\pi^{\vee}$ in 
 $\dualcat(\OO)$ and let $E=\End_{\dualcat(\OO)}(P)$. Then $\pi\hookrightarrow P^{\vee}$ is an injective envelope of $\pi$ in $\Mod^{\ladm}_{G}(\OO)$. 
The following result is \cite[cor.~3.10]{eff}.

\begin{prop}\label{projective}
The restriction of $P^{\vee}$ to $K$ is injective in $\Mod^{\sm}_K(\OO)$, hence $P$ is projective 
in $\Mod^{\pro}_K(\OO)$. 
\end{prop}

  In particular\footnote{Alternatively one may argue in the same way as in \cite[cor.~5.19]{cmf}.}, $P$ is a torsionfree, compact linear-topological $\OO$-module. 
 Let 
 $$\Pi(P):=\Hom_{\OO}^{\rm cont}(P, L)$$ 
 with the topology induced by the supremum 
norm. If $\Pi$ is an $L$-Banach space and if
$\Theta$ is an open, bounded lattice in $\Pi$, let $\Theta^d:=\Hom_{\OO}(\Theta, \OO)$ be its Schikhof dual. Equipped with the topology of pointwise convergence,  $\Theta^d$ is a torsionfree, compact linear-topological
 $\OO$-module and it follows from \cite[Thm.1.2]{iw} that 
 we have a natural isomorphism:
\begin{equation}\label{iso_zero} 
\Hom_L^{\rm cont}(\Pi, \Pi(P))\cong \Hom_{\OO}^{\rm cont}(P, \Theta^d)\otimes_{\OO} L.
\end{equation}
We want to use \eqref{iso_zero} in two ways, which are consequences of \cite[Thm.2.3]{iw}. If $\Pi$ is an admissible unitary $L$-Banach representation of $G$ and $\Theta$ is an open, bounded, $G$-invariant lattice in $\Pi$, then $\Theta^d$ is in $\dualcat(\OO)$ and we have: 
\begin{equation}\label{way_one}
\Hom_G^{\rm cont}(\Pi, \Pi(P))\cong \Hom_{\dualcat(\OO)}(P, \Theta^d)\otimes_{\OO} L=\md(\Pi)
\end{equation}

On the other hand, if $V$ is a continuous representation of $K$ on a finite dimensional $L$-vector space and if $\Theta$ is a $K$-invariant lattice in $V$, then 
\begin{equation}\label{way_two}
\Hom_K(V, \Pi(P))\cong \Hom_{\OO\br{K}}^{\rm cont}(P, \Theta^d)\otimes_{\OO} L\cong \Hom^{\rm cont}_{\OO\br{K}}(P, V^*).
\end{equation} 
We note that since $V$ is finite dimensional any $L$-linear map is continuous.

    Let ${\rm Alg}(G)$ be the set of isomorphism classes of irreducible rational representations of $\GL_2/L$, which we view as representations of $\GL_2(\Zp)$ via the inclusion $\GL_2(\Zp)\subset \GL_2(L)$. For $V\in {\rm Alg}(G)$ let $A_V:=\End_G(\cIndu{K}{G}{V})$. It follows from 
\cite[rem.2.1.4.2]{breuil2} that $A_V\cong \End_G(\cIndu{K}{G}{\Eins})\cong L[t, z^{\pm1}]$. In particular, $A_V$ is a commutative noetherian 
ring. Frobenius reciprocity gives  $$\Hom_K(V, \Pi(P))\cong \Hom_G(\cIndu{K}{G}{V}, \Pi(P)).$$ Hence, $\Hom_K(V, \Pi(P))$  is naturally an $A_V$-module. 
  We transport the action of $A_V$ onto $\Hom^{\rm cont}_{\OO\br{K}}(P, V^*)$ via \eqref{way_two}.

\begin{prop}\label{crystalline} Let $V\in {\rm Alg}(G)$ and let $\mm$ be a maximal ideal of $A_V$. Then $$ \dim_{\kappa(\mm)} \Hom_G(\kappa(\mm)\otimes_{A_V} \cIndu{K}{G}{V}, \Pi(P))\leq d(\pi).$$
\end{prop}

\begin{proof} It follows from \cite[prop.~3.2.1]{breuil2} that 
\begin{equation}\label{principal_series}
\kappa(\mm)\otimes_{A_V} \cIndu{K}{G}{V}\cong (\Indu{B}{G}{\delta_1\otimes\delta_2|\cdot|^{-1}})_{\sm}\otimes_L V,
\end{equation}
where $\delta_1, \delta_2: \Qp^{\times}\rightarrow \kappa(\mm)^{\times}$ are unramified characters with $\delta_1|\cdot|\neq \delta_2$ and the subscript $\sm$ indicates smooth induction. Let $\Pi$ the universal unitary completion of 
$\kappa(\mm)\otimes_{A_V} \cIndu{K}{G}{V}$. Since the action of $G$ on $\Pi(P)$ is unitary, the universal property of $\Pi$ implies that 
 \begin{equation}\label{universal_property}
 \Hom_G(\kappa(\mm)\otimes_{A_V} \cIndu{K}{G}{V}, \Pi(P))\cong \Hom_G^{\rm cont}(\Pi, \Pi(P))\overset{\eqref{way_one}}\cong \md(\Pi).
 \end{equation}
   It is proved in \cite[prop.~2.10]{blocks}, using 
results of Berger--Breuil \cite{bb} as the main input, that $\Pi$ is an admissible finite length $\kappa(\mm)$-Banach representation of $G$. Moreover, if $\Pi$  is non-zero then $\overline{\Pi}^{ss}$ is either  irreducible supersingular, or $\overline{\Pi}^{ss}\subseteq \pi\{\chi_1, \chi_2\}$ for some smooth characters $\chi_1, \chi_2: \Qp^{\times}\rightarrow k_{\kappa(\mm)}^{\times}$. Lemma~\ref{ss} implies that $\pi\otimes_k k_{\kappa(\mm)}$ can occur in $\overline{\Pi}^{ss}$ with multiplicity at most $d(\pi)$. Hence, if $\Theta$ is an open, bounded and $G$-invariant lattice in $\Pi$, then $\pi$ can occur as a subquotient of $\Theta/(\varpi)$ with multiplicity at most $ [\kappa(\mm): L] d(\pi)$. 
Proposition~\ref{easy} (ii) yields $\dim_L  \md(\Pi)\le [\kappa(\mm): L] d(\pi)$. The result follows from \eqref{universal_property}.
\end{proof}

 \begin{cor}\label{dual2} For all $V\in {\rm Alg}(G)$ and all maximal ideals $\mm$ of $A_V$ we have  
$$\dim_{\kappa(\mm)}\Hom_{\OO\br{K}}^{\rm cont}(P, V^*)[\mm]\leq d(\pi).$$ 
\end{cor}
\begin{proof} By \eqref{way_two} we have $\Hom_{\OO\br{K}}^{\rm cont}(P, V^*)[\mm]\cong\Hom_K(V, \Pi(P))[\mm]$. On the other hand, Frobenius reciprocity gives an isomorphism  
 \begin{equation}\label{four}
 \begin{split}
 \Hom_K(V, \Pi(P))[\mm]&\cong  \Hom_G(\cIndu{K}{G}{V}, \Pi(P))[\mm]\\
 &\cong \Hom_G(\kappa(\mm)\otimes_{A_V} \cIndu{K}{G}{V}, \Pi(P)).
 \end{split}
\end{equation}
The result follows therefore from proposition~\ref{crystalline}. 
\end{proof}

\subsection{Proof of theorem~\ref{I}}

\begin{prop}\label{Omega} Let $\varphi: P\twoheadrightarrow M$ be a quotient in $\dualcat(\OO)$, such that $M$ is of finite length. 
Then $\varphi$ factors through $\psi: P\rightarrow N$ in $\dualcat(\OO)$, such that $N$ is a finitely generated projective $\OO\br{K}$-module.
\end{prop}
\begin{proof} We claim that there exists a surjection $\theta: N\twoheadrightarrow M$ in $\dualcat(\OO)$ with $N$  a finitely generated projective $\OO\br{K}$-module.
The claim implies the assertion, since the projectivity of $P$ implies that there exists $\psi: P\rightarrow N$, such that $\theta\circ \psi=\varphi$.
The proof of the claim is a variation of the construction, which first appeared in  \cite{coeff}, and then was generalized in \cite{bp} and \cite{eff}. Let 
$G^0=\{g\in G: \det g\in \Zp^{\times}\}$ and let $G^+= Z G^0$, where $Z$ is the centre of $G$. Since $M^{\vee}$ is of finite length 
in $\Mod^{\ladm}_G(\OO)$, $M^{\vee}$ is admissible. It follows from \cite[Thm.3.4]{eff} that there
exists an injection $M^{\vee}\hookrightarrow \Omega$ in $\Mod^{\adm}_{G^0}(\OO)$, such that $M^{\vee}|_K \hookrightarrow \Omega|_K$ is 
an injective envelope of $M^{\vee}$ in $\Mod^{\sm}_K(\OO)$ and $\Omega\cong \Omega^c$, where $\Omega^c$ denotes the 
action of $G^0$ twisted by conjugation with an element $\bigl(\begin{smallmatrix} 0 & 1 \\ p & 0\end{smallmatrix}\bigr)$.
Dually we obtain a continuous, $G^0$-equivariant surjection 
$\theta^0:\Omega^{\vee}\twoheadrightarrow M$, such that its restriction to $K$  is a projective envelope of $M$ in $\Mod^{\pro}_K(\OO)$. 

We let  $A:=\OO[t, t^{-1}]$ act on $M$ by letting $t$ act as the matrix $\bigl (\begin{smallmatrix} p & 0 \\0 & p\end{smallmatrix} \bigr )$. 
Since $M$ is a quotient of $P$, its  cosocle in $\dualcat(\OO)$  is isomorphic to $\pi^{\vee}$, and hence is irreducible. This implies that 
$M$ is indecomposable. Moreover, $M$ is finite length by assumption. The argument of \cite[cor.~3.9]{eff} shows that there exists a monic 
polynomial $f\in \OO[t]$ and a natural number $n$, such that  $(\varpi, f)$ is a maximal ideal of $A$, and the action of $A$ on $M$ factors through $A/(f^n)$. 
Since $f$ is monic $A/(f^n)$ is a free $\OO$-module of finite rank. Hence, the restriction of $N^+:=A/(f^n)\otimes_{\OO} \Omega^{\vee}$ to 
$K$ is a finite direct sum of copies of $\Omega^{\vee}$, which implies that $N^+$ is a finitely generated projective $\OO\br{K}$-module.
We put an action of $G^+$ on $N^+$ by using $G^+= \bigl (\begin{smallmatrix} p & 0 \\0 & p\end{smallmatrix} \bigr )^{\ZZ}\times G^0$.
The map  $t\otimes v \mapsto \theta^0(\bigl (\begin{smallmatrix} p & 0 \\0 & p\end{smallmatrix} \bigr ) v)$ induces a $G^+$-equivariant surjection 
$\theta^+: N^+\twoheadrightarrow M$. Let $N:=\Indu{G^+}{G} N^+$, then by Frobenius reciprocity we obtain a surjective map 
$\theta: N\twoheadrightarrow M$. Since $G^+$ is of index $2$ in $G$,  and $\bigl(\begin{smallmatrix} 0 & 1 \\ p & 0\end{smallmatrix}\bigr)$
is a representative of the non-trivial coset, we have $N|_{G^+}\cong N^+\oplus (N^+)^c \cong N^+\oplus N^+$, where the subscript $c$ indicates that the 
action of $G^+$ is twisted by conjugation with $\bigl(\begin{smallmatrix} 0 & 1 \\ p & 0\end{smallmatrix}\bigr)$, and the last isomorphism follows from 
$\Omega\cong \Omega^c$. Hence, $N$ satisfies the conditions of the claim.
\end{proof}

 If $V$ is a continuous representation of $K$ on a finite dimensional $L$--vector space and if $\Theta$ is an open, bounded and $K$-invariant lattice in $V$, let $|\centerdot|$ be the norm
on $V^*$ given by $|\ell|:= \sup_{v\in \Theta} |\ell(v)|$, so that $\Theta^d=\Hom_{\OO}(\Theta, \OO)$ is the unit ball in $V^*$ with respect to $|\centerdot|$.
The topology on $\Hom_{\OO\br{K}}^{\rm cont}(P, V^*)$ is given by the norm $\|\phi\|:=\sup_{v\in P} |\phi(v)|$, and $\Hom_{\OO\br{K}}^{\rm cont}(P, \Theta^d)$ is the 
unit ball in this Banach space.

\begin{prop}\label{dense} For all $V$ as above the submodule  $$\Hom_{\OO\br{K}}^{\rm cont}(P, V^*)_{\mathrm{l.fin}}:=\{\phi\in \Hom_{\OO\br{K}}^{\rm cont}(P, V^*): \ell_{A_V}(A_V \phi)< \infty\}$$ is 
dense in $\Hom_{\OO\br{K}}^{\rm cont}(P, V^*)$, where $\ell_{A_V}(A_V \phi)$ is the length of $A_V \phi$ as an $A_V$-module.
\end{prop}
\begin{proof} Let $A=A_V$. It is enough to show that for each $\phi\in \Hom_{\OO\br{K}}^{\rm cont}(P, \Theta^d)$ and each $n\ge 1$ there exists $\psi_n\in
\Hom_{\OO\br{K}}^{\rm cont}(P, \Theta^d)$ such that the $A$-submodule generated by $\psi_n$ is of finite length, and $\phi\equiv \psi_n\pmod{\varpi^n}$.

Let $\phi_n$ be the composition $P\overset{\phi}{\rightarrow }\Theta^d\twoheadrightarrow \Theta^d/(\varpi^n)$. Dually we obtain a map 
$\phi_n^{\vee}: (\Theta^d/(\varpi^n))^{\vee}\rightarrow P^{\vee}$. Let $\tau$ be the $G$-subrepresentation of $P^{\vee}$ generated by the image of $\phi_n^{\vee}$. Since $P^{\vee}$ is in $\Mod^{\ladm}_G(\OO)$ any finitely generated $G$-subrepresentation is of finite length. Since $ (\Theta^d/(\varpi^n))^{\vee}$ is a finite $\OO$-module, we deduce that
$\tau$ is of finite length.  Thus $\phi_n$ factors through $P\twoheadrightarrow \tau^{\vee}$ in $\dualcat(\OO)$, with $\tau^{\vee}$ of finite length. Proposition~\ref{Omega} implies that this map factors through $\psi: P\rightarrow N$ with $N$ finitely generated and projective $\OO\br{K}$-module. 
Since $N$ is projective,  using the exact sequence 
$0\rightarrow \Theta^d\overset{\varpi^n}{\rightarrow} \Theta^d\rightarrow \Theta^d/(\varpi^n)\rightarrow 0$, we deduce that there exists
$\theta_n\in \Hom_{\OO\br{K}}^{\rm cont}(N, \Theta^d)$, which maps to $\phi_n\in \Hom_{\OO\br{K}}^{\rm cont}(N, \Theta^d/(\varpi^n))$. Let $\psi_n=\theta_n \circ \psi$.
Then by construction $\phi\equiv \psi_n \pmod{\varpi^n}$. Since $\psi$ is $G$-equivariant, 
$\Hom_{\OO\br{K}}^{\rm cont}(N, V^*)\overset{\circ \psi}{\rightarrow}\Hom_{\OO\br{K}}^{\rm cont}(P, V^*)$ is a map of $A$-modules, which contains $\psi_n$ in 
the image. Since $N$ is a finitely generated $\OO\br{K}$-module, 
$\Hom_{\OO\br{K}}^{\rm cont}(N, V^*)$ is a finite dimensional $L$-vector space, thus the $A$-submodule generated by $\psi_n$ is of finite length.
\end{proof}

 \begin{cor}\label{standard} For  $V$ as above, let $\mathfrak a_V$ be the $E[1/p]$-annihilator of $\Hom_{\OO\br{K}}^{\rm cont}(P, V^*)$. Then $E[1/p]/\mathfrak a_V$ satisfies the standard 
identity $s_{2d(\pi)}$ (see definition~\ref{dofpi} for $d(\pi)$). 
\end{cor}
\begin{proof} Since the action of $E$ preserves the unit ball in $\Hom_{\OO\br{K}}^{\rm cont}(P, V^*)$, $E[1/p]$ acts by continuous endomorphisms, which 
commute with the action of $A_V$. It follows from proposition~\ref{dense} that $E[1/p]/\mathfrak a_V$ injects into $\End_{A_V}(
\Hom_{\OO\br{K}}^{\rm cont}(P, V^*)_{\mathrm{l.fin}})$. It follows from proposition~\ref{crystalline} and lemma~\ref{dual2} that 
$$\dim_{\kappa(\mm)} \Hom_{\OO\br{K}}^{\rm cont}(P, V^*)_{\mathrm{l.fin}}[\mm] \le d(\pi),$$
for every maximal ideal $\mm$  of $A_V$. The assertion follows from lemma~\ref{get_it_right}.
\end{proof}

\begin{thm}\label{main_res_fin} Let $\Pi$ be a unitary admissible absolutely irreducible $L$-Ba\-na\-ch space representation of $G$ and let $\Theta$ be an open bounded $G$-invariant lattice in $\Pi$. Then $\pi$ occurs with multiplicity $\leq d(\pi)$ as a subquotient of $\Theta\otimes_{\OO} k$. 
\end{thm}
\begin{proof} Let $d=d(\pi)$, then it is enough to prove that $\dim_L \md(\Pi)\leq d$ by proposition~\ref{easy} (ii). 
 It follows from propositions~\ref{projective} and~\ref{capture_algebraic} that $\Alg(G)$ captures $P$, and lemma~\ref{intersect} (ii) implies that $\bigcap_{V\in {\Alg}(G)} \mathfrak a_V=0$, where
$\mathfrak a_V$ is defined in corollary~\ref{standard}. 
 We deduce from corollary~\ref{standard} and remark~\ref{commute}
 that $E[1/p]$ satisfies the standard identity $s_{2d}$. Thus, if $\mathcal E$ is the image of $E[1/p]$ in $\End_L(\md(\Pi))$, then $\mathcal E$ satisfies the standard identity $s_{2d}$. 

Since $\Pi$ is  irreducible, it follows from proposition~\ref{easy}(iii)a) that 
$\md(\Pi)$ is an irreducible $\mathcal E$-module, which is clearly faithful. 
Proposition~\ref{easy}(iii)b) shows that $\End_{\mathcal E}(\md(\Pi))={\End}_{G}^{\rm cont}(\Pi)^{\rm op}$. On the other hand, since 
$\Pi$ is absolutely irreducible, 
Schur's lemma \cite[Thm.1.1.1]{DS} yields $\End_G^{\rm cont}(\Pi)=L$, hence $\End_{\mathcal E}(\md(\Pi))=L$. 
A theorem of Kaplansky, see \cite[Thm.II.1.1]{processi} and \cite[cor.~II.1.2]{processi},  implies that $\dim_L \md(\Pi)\le d$, which is the desired result. \end{proof}

\begin{cor}\label{asserts} Let $\pi$ be an absolutely irreducible smooth representation and let $P\twoheadrightarrow \pi^{\vee}$ be a projective envelope 
of $\pi^{\vee}$ in $\dualcat(\OO)$, where $\dualcat(\OO)$ is the Pontryagin dual  of $\Mod^{\ladm}_G(\OO)$. If one of the following holds:
\begin{itemize}
\item[(i)] $\pi$ is supersingular;
\item[(ii)] $\pi\cong (\Indu{B}{G}{\chi_1\otimes \chi_2\omega^{-1}})_{\sm}$ and $\chi_1\chi_2^{-1}\neq \omega^{\pm 1}$, $\Eins$;
\item[(iii)] $\pi\cong (\Indu{B}{G}{\chi\omega\otimes \chi\omega^{-1}})_{\sm}$ and $p\ge 5$;
\item[(iv)] $\pi\cong \Sp\otimes\chi\circ \det$ and $p\ge 3$;
\item[(v)] $\pi\cong \chi\circ \det$ and $p\ge 3$;
\end{itemize} 
then the ring $E:=\End_{\dualcat(\OO)}(P)$ is commutative.
\end{cor}
\begin{proof} In these cases $d(\pi)=1$, and the assertion follows from the proof of 
theorem~\ref{main_res_fin} and remark~\ref{commute}.
\end{proof}

\begin{cor}\label{redBan} Let $\Pi$ be a unitary admissible absolutely irreducible $L$-Ba\-na\-ch space representation of $G$ and let $\Theta$ be an open bounded $G$-invariant lattice in $\Pi$. Then $\Theta\otimes_{\OO} k$ 
is of finite length. Moreover, one of the following holds: 
\begin{itemize}
\item[(i)] $\Theta\otimes_{\OO} k$ is absolutely irreducible supersingular;
\item[(ii)] $\Theta\otimes_{\OO} k$ is irreducible and 
$$\Theta\otimes_{\OO} l\cong (\Indu{P}{G}{\chi\otimes \chi^{\sigma} \omega^{-1}})_{\sm}\oplus (\Indu{P}{G}{\chi^{\sigma}\otimes \chi \omega^{-1}})_{\sm}$$
where $l$ is a quadratic extension of $k$, $\chi:\Qp^{\times}\rightarrow l^{\times}$ a smooth character and $\chi^{\sigma}$ is a conjugate of 
$\chi$ by the non-trivial element in $\Gal(l/k)$;
\item[(iii)] $(\Theta\otimes_{\OO} k)^{ss}\subseteq \pi\{\chi_1,\chi_2\}$ for some smooth characters $\chi_1, \chi_2: \Qp^{\times} \rightarrow k^{\times}$.
\end{itemize}
\end{cor}

\begin{proof} Let $\pi$ be an irreducible subquotient of $\Theta\otimes_{\OO} k$. If $\pi'$ is another  irreducible subquotient of $\Theta\otimes_{\OO} k$ then $\pi$ and $\pi'$ lie in the same block by \cite[prop. 5.36]{cmf}, which means that there exist irreducible smooth 
$k$-representations $\pi=\pi_0, \ldots, \pi_n=\pi'$, such that for all $0\le i <n$ either $\Ext^1_G(\pi_i, \pi_{i+1})\neq 0$ or 
$\Ext^1_G(\pi_{i+1}, \pi_i)\neq 0$. The blocks containing an absolutely irreducible representation have been determined in \cite{blocks}, 
and consist of either a single supersingular representation, or of all irreducible subquotients of $\pi\{\chi_1,\chi_2\}$ for some smooth characters $\chi_1, \chi_2: \Qp^{\times} \rightarrow k^{\times}$. 
These irreducible subquotients  are listed explicitly in lemma~\ref{ss}. If $\pi$ is absolutely irreducible, it follows from theorem~\ref{main_res_fin} that 
if $\pi$ is supersingular then (i) holds, if $\pi$ is not supersingular then the multiplicity with which $\pi$ occurs as a subquotient of 
$\Theta\otimes_{\OO} k$ is less or equal to the multiplicity with which $\pi$ occurs in $\pi\{\chi_1,\chi_2\}$, which implies that (iii) holds. If $\pi$ is not absolutely irreducible, 
then arguing as in the proof of corollary 5.44 of \cite{cmf} we deduce that (ii) holds.
\end{proof}

\section{Injectivity of the functor $\Pi\mapsto{\bf V}(\Pi)$} 
 In this chapter we prove theorems~\ref{II} and~\ref{nonordinary} as well as 
their consequences stated in the introduction. 
 After a few preliminaries devoted to the theory of $(\varphi,\Gamma)$-modules and various constructions involved
 in the $p$-adic local Langlands correspondence \cite{Cbigone}, we give a detailed overview of the (rather technical) proofs. 
 We then go on and supply the technical details of the proofs. 
  
 \subsection{Preliminaries}
 
 \subsubsection{$(\varphi,\Gamma)$-modules}   Let
   $\oe$ be the $p$-adic completion of 
    $\O [[T]][T^{-1}]$, $\mathcal{E}=\oe [p^{-1}]$ the field of fractions of $\oe$ and let $\mathcal{R}$ be the Robba ring, consisting of those Laurent series $\sum_{n\in \mathbb{Z}} a_nT^n\in L[[T,T^{-1}]]$ which converge on some annulus 
    $0<v_p(T)\leq r$, where $r>0$ depends on the series. 
      
      Let $\Phi\Gamma^{\rm et}(\mathcal{E})$ be the category of \'etale $(\varphi,\Gamma)$-modules over $\mathcal{E}$. 
      These are finite dimensional $\mathcal{E}$-vector spaces $D$ endowed with semi-linear\footnote{The rings 
      $\oe, \mathcal{E}, \mathcal{R}$ are endowed with a Frobenius $\varphi$ and an action of $\Gamma$ defined by 
      $\varphi(T)=(1+T)^p-1$ and $\gamma(T)=(1+T)^{\varepsilon(\gamma)}-1$.}  and commuting actions of 
      $\varphi$ and $\Gamma={\rm Gal}(\qp(\mu_{p^{\infty}})/\qp)\cong \zpet$ such that the action of $\varphi$ is \'etale\footnote{This means that the matrix of
      $\varphi$ in some basis of $D$ belongs to 
      ${\rm GL}_d(\oe)$, where $d=\dim_{\mathcal{E}} (D)$.}. Each $D\in\Phi\Gamma^{\rm et}(\mathcal{E})$
      is naturally endowed with an operator $\psi$, which is left-inverse to $\varphi$ and commutes with $\Gamma$. 
            
       The category $\Phi\Gamma^{\rm et}(\mathcal{E})$ is equivalent~\cite{FoGrot} to the category
      ${\rm Rep}_{L}(\cal G_{\qp})$ of continuous finite dimensional $L$-representations of ${\cal G}_{\qp}$. Cartier duality\footnote{Sending $V$ to $\check{V}:=V^*\otimes\varepsilon$, where $V^*$ is the $L$-dual of $V$.} on ${\rm Rep}_{L}(\cal G_{\qp})$ induces a Cartier duality $D\to \check{D}$ on $\Phi\Gamma^{\rm et}(\mathcal{E})$.
   All these constructions have integral and torsion 
      analogues, which will be used without further comment.

           The category $\Phi\Gamma^{\rm et}(\mathcal{E})$ is equivalent, by \cite{CCsurconv,BC07} and \cite{KK1}, 
        to the category of $(\varphi,\Gamma)$-modules of slope $0$ on $\mathcal{R}$. For $D\in\Phi\Gamma^{\rm et}(\mathcal{E})$
        we let $D_{\rm rig}$ be the associated $(\varphi,\Gamma)$-module over $\mathcal{R}$.
                     
               \subsubsection{ Analytic operations on $(\varphi,\Gamma)$-modules}\label{analytic}
               
                     The monoid $P^+=\left(\begin{smallmatrix} \zp-\{0\} & \zp \\0 & 1\end{smallmatrix}\right)$ acts naturally on 
        $\zp$ by $\left(\begin{smallmatrix} a & b \\0 & 1\end{smallmatrix}\right)x=ax+b$. Any $D\in \Phi\Gamma^{\rm et}(\mathcal{E})$
      carries a $P^+$ action, defined by 
      $$ \left(\begin{smallmatrix} p^ka & b \\0 & 1\end{smallmatrix}\right)z=(1+T)^b\varphi^k\circ \sigma_a(z)$$
      for $a\in\zpet, b\in \zp$ and $k\in\mathbb{N}$. 
      
      $D$ also gives rise to a $P^+$-equivariant sheaf $U\mapsto D\boxtimes U$ on 
      $\zp$, whose sections on $i+p^k\zp$ are $\left(\begin{smallmatrix} p^k & i \\0 & 1\end{smallmatrix}\right)D\subset D=D\boxtimes\zp$ 
      and for which the restriction map ${\rm Res}_{i+p^k\zp}: D\boxtimes\zp\to D\boxtimes (i+p^k\zp)$ is given by 
      $\left(\begin{smallmatrix} 1 & i \\0 & 1\end{smallmatrix}\right)\circ \varphi^k\circ \psi^k\circ \left(\begin{smallmatrix} 1 & -i \\0 & 1\end{smallmatrix}\right)$. 
      
 Let $U$ be an open compact subset of $\zp$ and let 
       $\phi: U\to L$ be a continuous function. By \cite[prop. V.2.1]{Cmirab}, the limit
       $$m_{\phi}(z)=\lim_{N\to\infty} \sum_{i\in I_N(U)} \phi(i) {\rm Res}_{i+p^N\zp}(z)$$
       exists for all $z\in D\boxtimes U$, and it is independent of the system of representatives $I_N(U)$ of $U$ mod $p^N$. Moreover, the resulting map 
       $m_{\phi}: D\boxtimes U\to D\boxtimes U$ is $L$-linear and continuous. 
              
        In the same vein \cite[prop. V.1.3]{Cmirab}, if $U,V$ are compact open subsets of $\zp$ and if $f: U\to V$ is a local diffeomorphism, 
       there is a direct image operator
        $$f_*: D\boxtimes U\to D\boxtimes V, \quad f_*(z)=\lim_{N\to\infty} \sum_{i\in I_N(U)} \left(\begin{smallmatrix} f'(i) & f(i) \\0 & 1\end{smallmatrix}\right){\rm Res}_{p^n\zp}
        (\left(\begin{smallmatrix} 1 & -i \\0 & 1\end{smallmatrix}\right)z).$$
        The following result (see $\S$ $V.1$ and $V.2$ in \cite{Cmirab}) summarizes the main properties of these operators (which also have integral and torsion versions, see
        loc.cit.).

          \begin{prop} \label{mult} Let $U,V$ be compact open subsets of $\zp$.
          
          a) For all continuous maps $\phi_1, \phi_2: U\to L$ we have $m_{\phi_1}\circ m_{\phi_2}=m_{\phi_1\phi_2}$.
          
          b) If $f: U\to V$ is a local diffeomorphism and $\phi: V\to L$ is continuous, then $$f_*\circ m_{\phi\circ f}=m_{\phi}\circ f_*.$$
          
          c) If $f: U\to V$ and $g: V\to W$ are local diffeomorphisms, then $g_*\circ f_*=(g\circ f)_*$. 
          
          d) If $\phi: U\to L$ is a continuous map and 
          $V\subset U$, then $m_{\phi}$ commutes with ${\rm Res}_V$.
          
          e) If $\phi:\zpet\to L$ is constant on $a+p^n\zp$ for all $a\in \zpet$, then 
          $$m_{\phi}=\sum_{i\in (\mathbb{Z}/p^n\mathbb{Z})^{\times}} \phi(i) {\rm Res}_{i+p^n\zp}.$$
          
                 \end{prop}

      \subsubsection{From ${\rm GL}_2(\qp)$ to ${\rm Gal}(\overline{\qp}/\qp)$-representations and back} 
      
       Let $G={\rm GL}_2(\qp)$. We refer the reader to the introduction for the definition of the category ${\rm Rep}_L(G)$, and to \cite[ch. IV]{Cbigone} (or to \cite[\S~III.2]{CD} for a summary) for the construction and study of an exact and contravariant functor
 $\mathbf{D}: {\rm Rep}_L(G)\to \Phi\Gamma^{\rm et}(\mathcal{E})$. Composing this functor with Fontaine's \cite{FoGrot} equivalence of categories
and Cartier duality, we obtain an exact covariant functor $\Pi\mapsto{\bf V}(\Pi)$ 
from ${\rm Rep}_L(G)$ to ${\rm Rep}_L(G_{\qp})$. 
We will actually work with the functor $\mathbf{D}$, even though some results will be stated in terms of the more familiar functor $\mathbf{V}$. 
    
    In the opposite direction, $\delta$ being fixed, there is a functor
from $\Phi\Gamma^{\rm et}(\mathcal{E})$ to the category of 
    $G$-equivariant sheaves of topological $L$-vector spaces on $\p1=\p1(\qp)$
(the space of sections on an open set $U$ of $\p1$ of the sheaf associated to $D$ is denoted by
$D\boxtimes_{\delta} U$). If $D\in \Phi\Gamma^{\rm et}(\mathcal{E})$, then 
    the restriction to $\zp$ of the sheaf $U\mapsto D\boxtimes_{\delta} U$ is the $P^+$-equivariant sheaf attached to $D$ as in ${\rm n}^{\rm o}$~\ref{analytic} (in particular it does not depend on $\delta$).
    The space $D\boxtimes_{\delta}\p1$ of global sections of the sheaf attached to $D$ and $\delta$ is naturally a topological $G$-module. 
    
     \begin{defi} If $\delta:\qpet\to \O^{\times}$ is a unitary character then we let ${\rm Rep}_L(\delta)$ be the 
    full subcategory of ${\rm Rep}_L(G)$ consisting of representations with central character $\delta$ and we let
    $\cal{MF}(\delta)$ be the essential image of $\mathbf{D}|_{{\rm Rep}_L(\delta)}$. 
    \end{defi}

    The following result
     follows by combining the main results of \cite[chap. III]{CD}.
         
       \begin{prop}\label{compatible} 
If $\delta:\qpet\to \cal O^{\times}$ is a unitary character then
 there is a functor $$\mathcal{MF}(\delta^{-1})\to {\rm Rep}_L(\delta),\quad D\mapsto \Pi_{\delta}(D)$$
       such that for all $D\in \mathcal{MF}(\delta^{-1})$, we have:
       
       {\rm (i)} If $\eta$ is a unitary character, then\footnote{Here $D(\eta)$ is the $(\varphi,\Gamma)$-module
       obtained by twisting the action of $\varphi$ and $\Gamma$ by $\eta$.} $D(\eta)\in \mathcal{MF}(\eta^{-2}\delta^{-1})$ and 
       $$\Pi_{\eta^2\delta} (D(\eta))\cong \Pi_{\delta}(D)\otimes (\eta\circ\det).$$
       
      {\rm (ii)} $\check{D}\in \mathcal{MF}(\delta)$ and 
      there is an exact sequence\footnote{Of topological $G$-modules, where $ \Pi_{\delta^{-1}}(\check{D})^*$ is the weak dual of 
      $\Pi_{\delta^{-1}}(\check{D})$.}
              $$0\to \Pi_{\delta^{-1}}(\check{D})^*\to D\boxtimes_{\delta}\p1\to \Pi_{\delta}(D)\to 0.$$
        
       {\rm (iii)} There is a canonical isomorphism $\mathbf{D}(\Pi_{\delta}(D))\cong \check{D}$.
       
       {\rm (iv)}  If $\dim (D)\geq 2$, then $D$ is irreducible if and only if $\Pi_{\delta}(D)$ is irreducible.
         \end{prop}
         
      All these constructions have natural integral and torsion variants, which will be used without further comment: for instance, if 
         $D_0$ is an $\oe$-lattice in $D\in \mathcal{MF}(\delta^{-1})$ which is stable by $\varphi$ and $\Gamma$, then $\Pi_{\delta}(D_0)$ is an open, bounded
         and $G$-invariant lattice in $\Pi_{\delta}(D)$. 
            
The next result is the main ingredient for the proof of the surjectivity
of the $p$-adic Langlands correspondence for $G$ (cf.~\S~\ref{surjectif}).
Note that if $D\in \Phi\Gamma^{\rm et}(\mathcal{E})$,
 then $\det D$ corresponds by Fontaine's equivalence of categories
  to a continuous character of $\mathcal{G}_{\qp}$, which in turn can be seen as a unitary character $\det D: \qpet\to \mathcal{O}^{\times}$ by local class field theory. We define
  $$\delta_D: \qpet\to \mathcal{O}^{\times}, \quad \delta_D=\varepsilon^{-1}\det D.$$
  
\begin{prop}\label{padicL}
 If $D\in \Phi\Gamma^{\rm et}(\mathcal{E})$ is $2$-dimensional, then 
 $D\in \mathcal{MF}(\delta_D^{-1})$. 
\end{prop}

\begin{proof}
This is a restatement of \cite[prop.~10.1]{PCDnew}.
\end{proof}

     \subsection{Uniqueness of the central character}\label{overviewCD}
       
In this $\S$ we explain the steps of the proof of the following theorem, which is the main 
result of this chapter. 
                     
       \begin{thm}\label{main1}
        Let $D\in\Phi\Gamma^{\rm et}(\mathcal{E})$ be absolutely irreducible, $2$-dimensional. 
\linebreak
        If $D\in \mathcal{MF}(\delta^{-1})$ for some unitary character $\delta$, then 
        $\delta=\delta_D$. 
              \end{thm}
        
     For the rest of this $\S$ we let 
 $D\in \Phi\Gamma^{\rm et}(\mathcal{E})$ be as in theorem~\ref{main1}. 
 Let $D^+$ be the set of $z\in D$ such that the sequence 
$(\varphi^n(z))_{n\geq 0}$ is bounded in $D$. The module $D^+$ is the largest finitely generated $\mathcal{E}^+$-submodule of 
$D$, stable under $\varphi$ and $\Gamma$. 
We say that $D$ is \textit{of finite height} if $D^+$ spans $D$ as $\mathcal{E}$-vector space or, equivalently (since 
$D$ is irreducible) if $D^+\ne \{0\}$. The classification of representations of finite height 
         given in \cite{Berger} shows that if $D$ is of finite height, then 
         $D$ is trianguline\footnote{Actually, more is true but will not be needed in the sequel: $D$ is of finite height if and only if $D\in \mathcal{S}_{*}^{\rm cris}$, where $ \mathcal{S}_{*}^{\rm cris}$ is defined in ${\rm n}^{\rm o}$~\ref{prelimtriang}.}, so it suffices to treat the cases ``$D$ trianguline" and ``$D^+=\{0\}$".

 \subsubsection{The trianguline case}      Suppose that $D$ is trianguline and recall that 
        $\delta_D=\varepsilon^{-1}\det D$
     is seen as a unitary character of $\qpet$. Suppose that $D\in \mathcal{MF}(\delta^{-1})$ for some unitary character $\delta$. 
      Let $\Pi=\Pi_{\delta}(D)$ and define $ \eta=\delta_D^{-1}\delta$. Also, let 
      $\Pi^{\rm an}$ be the space of locally analytic vectors 
   of $\Pi$, for which we refer the reader to \cite{adm} and~\cite{em1}.

        We first prove that $\eta$ is locally constant. The argument works for a much more general class of $(\varphi,\Gamma)$-modules, namely those
        corresponding to representations which are not $\mathbf{C}_p$-admissible\footnote{That is, de Rham with Hodge-Tate weights equal to $0$.} up to a twist
        (this condition is automatically satisfied by irreducible trianguline $(\varphi,\Gamma)$-modules \cite[prop. 4.6]{Ctrianguline}). 
        The proof uses two ingredients 
        
        $\bullet$ By \cite[chap. VI]{CD} and the hypothesis $D\in \mathcal{MF}(\delta^{-1})$, there is 
    a $G$-equivariant sheaf of locally analytic representations $U\mapsto D_{\rm rig}\boxtimes_{\delta} U$ attached to $(D_{\rm rig},\delta)$. 
 
    $\bullet$ Using \cite{ExtdR, Annalen}, one can describe the action 
of ${\rm Lie}({\rm GL}_2(\qp))$ on the space $D_{\rm rig}\boxtimes_{\delta}\p1$ 
of global sections of this sheaf.
The fact that $\eta$ is locally constant 
follows easily.
    
        The second part of the proof in the trianguline case consists in analyzing the module $D_{\rm rig}\boxtimes_{\delta}\p1$. 
More precisely, we prove (lemma~\ref{devisser}) that
if $0\to \mathcal{R}(\delta_1)\to D_{\rm rig}\to \mathcal{R}(\delta_2)\to 0$ is a triangulation of $D_{\rm rig}$, then this exact sequence extends to 
       an exact sequence of topological $G$-modules 
  $$0\to \mathcal{R}(\delta_1)\boxtimes_{\delta}\p1\to D_{\rm rig}\boxtimes_{\delta}\p1\to \mathcal{R}(\delta_2)\boxtimes_{\delta}\p1\to 0.$$
  Actually, once we know that $\eta$ is locally constant, the arguments of \cite{Cvectan, Jacquet} 
(where the case $\eta=1$ is treated) go through without any change. 
  
    Using the  description~\cite{Cvectan} 
of the Jordan-H\"older components of $\mathcal{R}(\delta_1)\boxtimes_{\delta}\p1$, we deduce (lemma~\ref{morphism})
   the existence of a 
morphism with finite dimensional kernel
   $B^{\rm an}(\delta_1,\eta\delta_2)\to \Pi^{\rm an}$, where $B^{\rm an}(\delta_1,\eta\delta_2)$ is the locally analytic
parabolic induction of the character $\eta\delta_2\otimes \varepsilon^{-1}\delta_1$. Finally, using universal completions,
we prove 
   that the morphism $B^{\rm an}(\delta_1,\eta\delta_2)\to \Pi^{\rm an}$ induces a nonzero 
morphism $\Pi_{\delta}(D')\to \Pi$ 
   for some $G$-compatible pair $(D',\delta)$, where $D'$ is a trianguline $(\varphi,\Gamma)$-module having a triangulation 
   $$0\to \mathcal{R}(\delta_1)\to D'_{\rm rig}\to \mathcal{R}(\eta\delta_2)\to 0.$$ This is the most technical part of the proof 
and uses results from \cite{PS,Cvectan,bb} (and~\cite{except} suitably extended for $p=2$). 
Since $D$ and $D'$ are irreducible, the representations $\Pi_{\delta}(D')$ and $\Pi$ are admissible and topologically irreducible, hence 
the morphism $\Pi_{\delta}(D')\to \Pi$ must be an isomorphism. 
Using parts (iii) and (iv) of proposition~\ref{compatible}, we deduce 
that $D'\cong D$. In particular $\det D_{\rm rig}=\det D'_{\rm rig}$, which yields 
$\delta_1\delta_2=\delta_1\delta_2\eta$, hence $\eta=1$,
 and finishes the proof in the trianguline case. 
 
     \subsubsection{The case $D^+=\{0\}$}  Let us assume now that $D\in \Phi\Gamma^{\rm et}(\mathcal{E})$ is $2$-dimensional and satisfies $D^+=\{0\}$
       (then $D$ is automatically absolutely irreducible).  
       For each $\alpha\in \mathcal{O}^{\times}$ let $$\mathcal{C}^{\alpha}=(1-\alpha\varphi)D^{\psi=\alpha}.$$
        If $D\in \mathcal{MF}(\chi^{-1})$ for some character $\chi:\qpet\to \mathcal{O}^{\times}$, then setting $\check{\Pi}=\Pi_{\chi^{-1}}(\check{D})$
     we have $\check{\Pi}^*\subset D\boxtimes_{\chi}\p1$ (proposition~\ref{compatible}), and there is \cite[rem.V.14(ii)]{CD} 
a canonical isomorphism of $\mathcal{O}[[\Gamma]][1/p]$-modules 
      \begin{equation}\label{Iwasawa}
{\rm Res}_{\zpet}((\check{\Pi}^*)^{\left(\begin{smallmatrix} p & 0\\0 & 1\end{smallmatrix}\right)=\alpha^{-1}})\cong \mathcal{C}^{\alpha}.
\end{equation}
      
Suppose now that $D\in\mathcal{MF}(\delta^{-1})$ and let
 $\eta=\delta_D^{-1}\delta$.
   Unravelling the isomorphism \eqref{Iwasawa} for 
   $\chi=\delta_D$ and $\chi=\delta$, we obtain the following key fact       
          
        \begin{prop}\label{crux}
       For all $\alpha\in \mathcal{O}^{\times}$ we have
        $m_{\eta}(\mathcal{C}^{\alpha})=\mathcal{C}^{\alpha\eta(p)}$. 
       \end{prop}
    
    \begin{proof}  
If $\chi\in\{\delta_D,\delta\}$, let $w_{\chi}$ be the restriction to $D^{\psi=0}=D\boxtimes_{\chi}\zpet$ of the action of 
    $w=\left(\begin{smallmatrix} 0 & 1\\1 & 0\end{smallmatrix}\right)$
    on $\Pi_{\chi^{-1}}(\check{D})^*\subset D\boxtimes_{\chi}\p1$. 
    Proposition V.12 of \cite{CD} shows that $w_{\chi}(\mathcal{C}^{\alpha})=\mathcal{C}^{\frac{1}{\alpha\chi(p)}}$
    for all $\alpha\in \mathcal{O}^{\times}$.
    On the other hand, remark II.1.3 of \cite{Cbigone} and part b) of proposition~\ref{mult} yield $w_{\chi}=w_*\circ m_{\chi}$ 
   and 
    $$w_{\delta_D}\circ w_{\delta}=w_*\circ m_{\delta_D}\circ w_*\circ m_{\delta}=m_{\delta_D^{-1}}\circ w_*\circ w_*\circ m_{\delta}=m_{\delta_D^{-1}\delta}=m_{\eta},$$
   as $w_*\circ w_*=(w\circ w)_*={\rm id}$.
    The result follows.
    \end{proof}
            
     In view of proposition~\ref{padicL}, theorem~\ref{main1}, in the case $D^+=0$, is equivalent to the following statement.
\begin{prop}\label{main1.1}
We have $\eta=1$.
\end{prop}
\begin{proof}       
     Let $\mu(\qp)$ be the set of roots of unity in $\qp$ and let 
       $$\hat{\mathcal{T}}^0(L)={\rm Hom}_{\rm cont}(\qpet/\mu(\qp), 1+{\mathfrak m}_L)$$ 
be the set of continuous characters $\chi:\qpet\to 1+{\mathfrak m}_L$ trivial on  $\mu(\qp)$, and let 
$$H=\{\chi\in \hat{\mathcal{T}}^0(L)|
      m_{\chi}(\mathcal{C}^{\alpha})=\mathcal{C}^{\alpha\chi(p)} \quad \forall  \alpha\in 1+{\mathfrak m}_L\}.$$
       Proposition~\ref{subgroup} below shows that $H$ is 
    a Zariski closed subgroup of $\hat{\mathcal{T}}^0(L)$ and it follows from 
     corollary~\ref{torsion} 
    that $H$ is either trivial or it contains a nontrivial character of finite order (this may require replacing $L$ by a finite extension, which we are allowed to do). 
    We haven't used so far the hypothesis $D^+=\{0\}$, but only the fact that $D$ is absolutely irreducible of dimension $\geq 2$. When $D^+=\{0\}$, we prove (corollary~\ref{trivial}) that $H$ cannot contain nontrivial locally constant characters. We conclude that $H=\{1\}$, which implies that 
    $\eta$ is of finite order (since any power of $\eta$ which belongs to $\hat{\mathcal{T}}^0(L)$ actually belongs to $H$ by proposition~\ref{crux}). We
    conclude that $\eta=1$ using again corollary~\ref{trivial}. 
\end{proof}     
             
\begin{remark}\label{Sen}  Assume that the Sen operator on $D$ is not scalar\footnote{By Sen's theorem, this is equivalent to saying that 
inertia does not have finite image on the Galois representation associated to $D$.}. Proposition~\ref{scalar} shows that 
  $\eta$ is locally constant, and proposition~\ref{crux} and corollary~\ref{trivial}
  yield directly the desired result $\eta=1$. The key proposition~\ref{scalar} does not work if the Sen operator is scalar,
 which explains the more indirect approach presented above. 
   \end{remark}

  \subsubsection{Consequences of theorem~\ref{main1}} 
   
    Before embarking on the proof of theorem~\ref{main1}, we 
 give a certain number of consequences of theorems~\ref{I} and~\ref{main1}. 

 If $D\in \Phi\Gamma^{\rm et}(\mathcal{E})$, we let $\overline{D}^{\rm ss}$ be the semi-simplification of 
      $D_0\otimes_{\O} k$, where $D_0$ is any $\oe$-lattice in $D$ which is stable under $\varphi$ and $\Gamma$.
      
      The functor $\Pi\mapsto\mathbf{V}(\Pi)$
has integral and torsion versions, and if $\Theta$ is an open, bounded and $G$-invariant lattice in $\Pi\in {\rm Rep}_L(G)$, then 
$\mathbf{V}(\Theta)=\varprojlim_{n} \mathbf{V}(\Theta/\varpi^n)$ and $\mathbf{V}(\Theta)/\varpi^n\cong \mathbf{V}(\Theta/\varpi^n)$ for all $n\geq 1$. 
 The following result follows from \cite[th. 0.10]{Cbigone} (see the introduction
for the definition of $\pi\{\chi_1,\chi_2\}$). 

\begin{lem}\label{vfortors}
 If $\pi$ is either supersingular or $\pi\{\chi_1,\chi_2\}$ for some smooth characters $\chi_1,\chi_2: \qpet\to k^{\times}$, then 
 $\dim_k \mathbf{V}(\pi)=2$. Moreover, if $\pi$ is an irreducible subrepresentation of $\pi\{\chi_1,\chi_2\}$, then 
 $\dim_k \mathbf{V}(\pi)\leq 1$.  
\end{lem}

  We will also need to following compatibility between the $p$-adic and mod $p$ Langlands correspondences. 
  This was first proved (in a slightly different form) in \cite{Becomp}. We will use the following version, taken from
   \cite[prop.~III.55, rem.~III.56]{CD}. 
  
  \begin{prop}\label{compat}
   If $D\in \Phi\Gamma^{\rm et}(\mathcal{E})$ is $2$-dimensional and if 
   $\delta=\delta_D$, there is an isomorphism
   $\overline{\Pi_{\delta}(D)}^{\rm ss}\cong \Pi_{\delta}(\overline{D}^{\rm ss})$
   and (possibly after replacing $L$ with its quadratic unramified extension) this representation is either
   absolutely irreducible supersingular or isomorphic to $\pi\{\chi_1,\chi_2\}$ for some
   smooth characters $\chi_1,\chi_2: \qpet\to k^{\times}$. 
   
  \end{prop}
 
     \begin{prop}\label{ordinary} 
     Let $\Pi\in {\rm Rep}_L(\delta)$ be absolutely irreducible and let $D=\mathbf{D}(\Pi)$. Then $\dim_{\mathcal{E}} D\leq 2$ and $D$ is 
     absolutely irreducible.
     Moreover, $\dim D=2$ if and only if $\Pi$ is non-ordinary. 
        \end{prop}
     
     \begin{proof}
      The functor $\mathbf{D}$ being exact, there is a natural isomorphism 
    $\overline{D}^{\rm ss}\cong \mathbf{D}(\overline{\Pi}^{\rm ss})$. Combined with theorem~\ref{I} and 
   lemma~\ref{vfortors}, this yields $\dim D\leq 2$. 
    
     Next, if $\Pi$ is ordinary, then 
    $\overline {\Pi}^{\rm ss}$ is a subquotient of a smooth parabolic induction and 
    using lemma~\ref{vfortors} again we conclude that $\dim D\leq 1$. In particular, $D$ is 
    absolutely irreducible.
    If $\Pi$ is not ordinary, we deduce from \cite[cor.~III.47]{CD} and the first paragraph that $D$ is absolutely irreducible and $2$-dimensional. The result follows. 
       \end{proof}

     \begin{cor}\label{determinant}
     If $\Pi\in {\rm Rep}_L(\delta)$ is absolutely irreducible non-ordinary, then 
     $\delta=\delta_{\mathbf{D}(\Pi)}$. Thus $\det \mathbf{V}(\Pi)=\varepsilon\delta$. 
     \end{cor}
    
    \begin{proof}
    This follows directly from proposition~\ref{ordinary} and theorem~\ref{main1}.
    \end{proof}                  
                            
         \begin{thm}\label{ssing}
          Let $\Pi\in {\rm Rep}_L(\delta)$ be absolutely irreducible and let $D=\mathbf{D}(\Pi)$. The following assertions are equivalent 
          
          {\rm (i)} $\Pi$ is non-ordinary. 
                 
        {\rm (ii)} $\dim_{\mathcal{E}} D=2$.
        
        {\rm (iii)} After possibly
replacing $L$ by its quadratic unramified extension, $\overline{\Pi}^{\rm ss}$ is either absolutely irreducible supersingular or isomorphic to some $\pi\{\chi_1,\chi_2\}$. 
       
          If these assertions hold, then there is a canonical isomorphism $\Pi\cong \Pi_{\delta}(\check{D})$. 
    
          \end{thm}
                 
     \begin{proof}
(i) and (ii) are equivalent by proposition~\ref{ordinary}. Suppose that (iii) holds. Then 
     $\overline{D}^{\rm ss}\cong \mathbf{D}(\overline{\Pi}^{\rm ss})$ is $2$-dimensional by lemma~\ref{vfortors} and 
     so $\dim D=2$, that is (ii) holds. Finally, suppose that (i) holds. Then \cite[cor.~III.47]{CD} yields a canonical isomorphism
$\Pi\cong \Pi_{\delta}(\check{D})$ and we conclude that (iii) holds
using proposition~\ref{compat}.     \end{proof}

  \begin{thm}
   Let $\Pi_1,\Pi_2\in {\rm Rep}_L(G)$ be absolutely irreducible, non-ordinary. 
   
   {\rm (i)} If $\VV(\Pi_1)\cong \VV(\Pi_2)$, then $\Pi_1\cong \Pi_2$.
   
   {\rm (ii)} We have ${\rm Hom}_{L[P]}^{\rm cont}(\Pi_1,\Pi_2)={\rm Hom}_{L[G]}^{\rm cont}(\Pi_1,\Pi_2)$, where $P$ is the mirabolic subgroup of $G$.
  \end{thm}

\begin{proof}
    If $\VV(\Pi_1)\cong \VV(\Pi_2)=V$, then corollary~\ref{determinant} shows that 
       $\Pi_1$ and $\Pi_2$ have the same central character $\delta$ and theorem~\ref{ssing}
 yields $\Pi_1\cong \Pi_2\cong\Pi_\delta({\bf D}(\check V))$.

        Let $f: \Pi_1\to \Pi_2$ be a $P$-equivariant linear continuous map
 and let $D_j=\mathbf{D}(\Pi_j)$. 
         Since the functor $\mathbf{D}$ uses only the restriction to $P$, $f$ induces a morphism
         $\mathbf{D}(f): D_2\to D_1$ in $\Phi\Gamma^{\rm et}(\mathcal{E})$. Since $D_1$ and $D_2$ are absolutely irreducible
         (proposition~\ref{ordinary}), $\mathbf{D}(f)$ is either $0$ or an isomorphism. If $\mathbf{D}(f)$ is an isomorphism, then $\delta_1=\delta_2$ by part (i), and we conclude that 
         $f$ is $G$-equivariant using 
         the following diagram, in which the vertical maps are the isomorphisms given by theorem~\ref{ssing} and the 
         map $\Pi_{\delta_1}(\check{D}_1)\to \Pi_{\delta_1}(\check{D}_2)$ is $G$-equivariant since induced by functoriality
         from the transpose of $\mathbf{D}(f)$
               $$\xymatrix{
& \Pi_1\ar[r]^f\ar[d]_{\cong}&\Pi_2 \ar[d]^{\cong}&\\
&\Pi_{\delta_1}(\check{D}_1)\ar[r]^{\mathbf{D}(f^*)}&\Pi_{\delta_1}(\check{D}_2).}$$ 

   The case $\mathbf{D}(f)=0$ is slightly trickier, since we can no longer use part (i) of the theorem to deduce that 
   $\Pi_1$ and $\Pi_2$ have the same central character. We will prove that $f=0$. 
   Let $\Theta_j$ be the unit ball in $\Pi_j$ and let 
   $X_j={\rm Res}_{\zp} (\Theta_j^{d})$, where we use the inclusions 
   $\Pi_j^{*}\subset D_j\boxtimes_{\delta_j^{-1}}\p1$ (proposition~\ref{compatible}). 
    It follows from \cite[cor.~III.25]{CD} that the restriction of
   ${\rm Res}_{\qp}: D_j\boxtimes_{\delta_j^{-1}}\p1\to D_j\boxtimes_{\delta_j^{-1}}\qp$ to $\Pi_j^{*}$ 
   induces a $P$-equivariant isomorphism of topological vector spaces
   $\Pi_j^{*}\cong (\varprojlim_{\psi} X_j)\otimes L$. We have a commutative diagram 
 $$\xymatrix{
& \Pi_2^*\ar[r]\ar[d]&\Pi_1^* \ar[d]&\\
&(\varprojlim_{\psi} X_2)\otimes L \ar[r]& (\varprojlim_{\psi} X_1)\otimes L&}$$   
in which the top horizontal map is the transpose $f^*$ of $f$, the vertical maps are the isomorphisms explained above and 
the horizontal map on the bottom is induced by $\mathbf{D}(f)=0$, and thus it is the zero map. We conclude that $f^*=0$ and thus 
$f=0$, which finishes the proof of theorem~\ref{II}. 
\end{proof}

  The remaining sections will be devoted to the proof of theorem~\ref{main1}
in the case $D$ trianguline (see proposition~\ref{Triangle}), 
and to the proof of the statements (namely proposition~\ref{subgroup} and corollaries~\ref{trivial} and~\ref{torsion})
that were used in the proof of proposition~\ref{main1.1}
which, as we remarked, is
equivalent to theorem~\ref{main1} in the case $D^+=0$.
                   
\subsection{Trianguline representations}\label{triangulinecase}\label{tri}

\subsubsection{Preliminaries}\label{prelimtriang}
    If $\delta:\qpet\to L^{\times}$ is a continuous character (not necessarily unitary), let $\mathcal{R}(\delta)$ be the 
   $(\varphi,\Gamma)$-module obtained by twisting the action of $\varphi$ and $\Gamma$ on $\mathcal{R}$
   by $\delta$. It has a canonical basis $e=1\otimes \delta$ for which $\varphi(e)=\delta(p)e$ and $\sigma_{a}(e)=\delta(a)e$
   for $a\in\zpet$, where $\sigma_a\in\Gamma$ satisfies $\sigma_a(\zeta)=\zeta^a$ for all $\zeta\in\mu_{p^{\infty}}$, so that
   $\varepsilon(\sigma_a)=a$. Let ${\cal R}^+$ be the ring of analytic
functions on the open unit disk,  so that ${\cal R}^+={\cal R}\cap L[[T]]$. We define ${\cal R}^+(\delta)$ as the ${\cal R}^+$-sumodule
of ${\cal R}(\delta)$ generated by $e$. We let  $\kappa(\delta)$ be the derivative of $\delta$ at $1$ or, equivalently (if $\delta$ is unitary),
 the generalized Hodge--Tate weight of the Galois character
 corresponding to $\delta$ by class field theory. 
   
   By\footnote{Contrary to \cite{Ctrianguline}, all results of \cite{Cvectan} are proved for all primes $p$.} \cite[prop. 0.2]{Cvectan}, 
    ${\rm Ext}^1(\mathcal{R}(\delta_2),\mathcal{R}(\delta_1))$ has dimension $1$ 
    when $\delta_1\delta_2^{-1}$ is not of the form 
 $x^{-i}$ or $\varepsilon x^{i}$ for some $i\geq 0$, and dimension $2$ in the remaining cases. 
   Moreover, if ${\rm Ext}^1(\mathcal{R}(\delta_2),\mathcal{R}(\delta_1))$ is $2$-dimensional, then 
   the associated projective space is naturally isomorphic to $\mathbf{P}^1(L)$. Let $\mathcal{S}$ be the set of triples $(\delta_1,\delta_2,\mathcal{L})$, where $\delta_1,\delta_2:\qpet\to L^{\times}$ are continuous characters and $\mathcal{L}\in {\rm Proj}({\rm Ext}^1(\mathcal{R}(\delta_2),\mathcal{R}(\delta_1)))$
 (if ${\rm Ext}^1(\mathcal{R}(\delta_2),\mathcal{R}(\delta_1))$ is $1$-dimensional, we have $\mathcal{L}=\infty$). Each $s\in\mathcal{S}$ gives rise to an 
 extension $0\to\mathcal{R}(\delta_1)\to \Delta(s)\to \mathcal{R}(\delta_2)\to 0$, classified up to isomorphism by $\mathcal{L}$.
 
  Let $\mathcal{S}_*$ be the subset of $\mathcal{S}$ consisting of those $s=(\delta_1,\delta_2,\mathcal{L})$ for which
  $v_p(\delta_1(p))+v_p(\delta_2(p))=0$ and $v_p(\delta_1(p))>0$. For each $s\in\mathcal{S}_*$ let
  $$u(s)=v_p(\delta_1(p)),\quad \kappa(s)=\kappa(\delta_1)-\kappa(\delta_2).$$
  Let $\mathcal{S}_*^{\rm cris}$ (resp. $\mathcal{S}_{*}^{\rm st}$) be the set of $s\in \mathcal{S}_*$
  for which $\kappa(s)\in\mathbf{N}^*$, $u(s)<\kappa(s)$ and $\mathcal{L}=\infty$ (resp. $\mathcal{L}\ne\infty$). 
  Let $\mathcal{S}_*^{\rm ng}$ be the set of $s\in\mathcal{S}_*$ for which $\kappa(s)\notin\mathbf{N}^*$ and finally let 
    \begin{center} $\mathcal{S}_{\rm irr}= \mathcal{S}_*^{\rm cris} \coprod \mathcal{S}_*^{\rm st}\coprod \mathcal{S}_*^{\rm ng}.$ \end{center}  
We say that $D\in \Phi\Gamma^{\rm et}(\mathcal{E})$ is trianguline
(of rank~$2$) if $D_{\rm rig}$ is an extension of two $(\varphi,\Gamma)$-modules
of rank~$1$ over ${\cal R}$. These are described by
the following result (\cite[prop. 0.3]{Cvectan} or \cite[th. 0.5]{Ctrianguline}). 
  
  \begin{prop}\label{classical}
  
a) For any $s\in\mathcal{S}_{\rm irr}$ 
there is a unique $D(s)\in\Phi\Gamma^{\rm et}(\mathcal{E})$ such that
$\Delta(s)=D(s)_{\rm rig}$, and $D(s)$ is absolutely irreducible.
Moreover, if $D\in \Phi\Gamma^{\rm et}(\mathcal{E})$ is trianguline
and absolutely irreducible, there exists $s\in \mathcal{S}_{\rm irr}$
such that $D\cong D(s)$.
    
   b) If $s=(\delta_1,\delta_2,\mathcal{L})$ and $s'=(\delta_1',\delta_2',\mathcal{L'})$ are elements of $\mathcal{S}_{\rm irr}$, then 
   $D(s)\cong D(s')$ if and only if $s,s'\in\mathcal{S}_{*}^{\rm cris}$ and $\delta_1'=x^{\kappa(s)}\delta_2$, $\delta_2'=x^{-\kappa(s)}\delta_1$.
  
  \end{prop}

  \subsubsection{Infinitesimal study of the module $D_{\rm rig}\boxtimes_{\delta}\p1$}
  
   In this $\S$ we let $D$ be any object of $\mathcal{MF}(\delta^{-1})$ for some unitary character $\delta$. 
 
   Recall (see section $2.5$ of \cite{Annalen} for a summary) that Sen's theory associates to $D$ a finite free $L\otimes_{\qp} \qp(\mu_{p^{\infty}})$-module $D_{\rm Sen}$ endowed with a Sen operator 
     $\Theta_{\rm Sen}$, whose eigenvalues are the generalized Hodge-Tate weights of $D$.     
 
\begin{prop}\label{scalar} If $D\in \mathcal{MF}(\delta^{-1})$ is absolutely irreducible, $2$-dimensional and if
$\Theta_{\rm Sen}$ is not a scalar operator on 
$D_{\rm Sen}$, then $\delta\delta_D^{-1}$ is locally constant. 
\end{prop}

\begin{proof}
 By 
 \cite[chap. VI]{CD},       
      the $G$-equivariant sheaf 
     $U\mapsto D\boxtimes_{\delta} U$ induces a $G$-equivariant sheaf $U\mapsto D_{\rm rig}\boxtimes_{\delta} U$  
     on $\mathbf{P}^1(\qp)$. We let $D_{\rm rig}\boxtimes_{\delta}\p1$ be the space of global sections of this sheaf.
     This is naturally an LF space and $G$ acts continuously on it. Moreover, this action extends to a structure of topological 
     $\mathcal{D}({\rm GL}_2(\zp))$-module on $D_{\rm rig}\boxtimes_{\delta}\p1$, where $\mathcal{D}({\rm GL}_2(\zp))$ is 
     the Fr\'echet-Stein algebra \cite{adm} of $L$-valued distributions on ${\rm GL}_2(\zp)$. In particular, the enveloping algebra of 
     $\mathfrak{gl}_2={\rm Lie}(G)$ acts on $D_{\rm rig}\boxtimes_{\delta}\p1$. 

    Consider the Casimir element 
    $$C=u^+u^-+u^-u^++\frac{1}{2}h^2\in U(\mathfrak{gl}_2),$$ where $h=\left(\begin{smallmatrix} 1 & 0 \\0 & -1\end{smallmatrix}\right)$, 
   $u^+=\left(\begin{smallmatrix} 0 & 1 \\0 & 0\end{smallmatrix}\right)$ and $u^-=\left(\begin{smallmatrix} 0 & 0 \\1 & 0\end{smallmatrix}\right)$. The action of $C$ on 
    $D_{\rm rig}\boxtimes_{\delta}\p1$ preserves $D_{\rm rig}=D_{\rm rig}\boxtimes_{\delta}\zp$, viewed as a sub-module of $D_{\rm rig}\boxtimes_{\delta}\p1$
via the extension by $0$.
  By  theorem 3.1 and remark 3.2 of \cite{ExtdR} the operator $C$ acts by a scalar $c$ on $D_{\rm rig}$ and we have an equality of operators on $D_{\rm Sen}$ 
     \begin{equation}\label{one1}
(2\Theta_{\rm Sen} -(1+\kappa(\delta)))^2=1+2c.
\end{equation}
    
    Let $a$ and $b$ be the generalized Hodge-Tate weights of $D$. By Cayley-Hamilton we have 
 $(\Theta_{\rm Sen}-a)(\Theta_{\rm Sen}-b)=0$
 as endomorphisms of $D_{\rm Sen}$. Combining this relation with \eqref{one1} yields
 $$4(a+b-1-\kappa(\delta)){\Theta}_{\rm Sen}+(1+\kappa(\delta))^2-4ab=1+2c.$$
    Since $\Theta_{\rm Sen}$ is not scalar, the previous relation forces $a+b-1=\kappa(\delta)$
and, as $a+b-1=\kappa(\delta_D)$, this gives $\kappa(\delta\delta_D^{-1})=0$. The result follows.
\end{proof}

\begin{remark}\label{locconstant} 
  If $D$ is trianguline, $2$-dimensional and irreducible, then 
  $\Theta_{\rm Sen}$ is not scalar, see proposition 4.6 of \cite{Ctrianguline}. 
\end{remark}

\subsubsection{D\'evissage of $D_{\rm rig}\boxtimes_{\delta}\p1$}

   If $\eta_1,\eta_2:\qpet\to L^{\times}$ are continuous characters, let $$B^{\rm an}(\eta_1,\eta_2)=({\rm Ind}_{B}^{G}\eta_2\otimes\varepsilon^{-1}\eta_1)^{\rm an}$$
   be the locally analytic parabolic induction of the character $\eta_2\otimes \varepsilon^{-1}\eta_1$.
  The recipe giving rise to the sheaf $U\mapsto D\boxtimes_{\delta} U$ for $D\in\Phi\Gamma^{\rm et}(\mathcal{E})$ can be used \cite[\S~3.1]{Jacquet} to create
  a $G$-equivariant sheaf 
  $U\mapsto \mathcal{R}(\eta_1)\boxtimes_{\delta} U$ on $\p1$, attached to the  
   pair $(\mathcal{R}(\eta_1), \delta)$. 
 We will only be interested 
   in the space $\mathcal{R}(\eta_1)\boxtimes_{\delta}\p1$ of its global sections, which is described by the following proposition, whose proof is easily deduced from remark~3.7 of \cite{Jacquet}.
Let
$$\mathcal{R}^+(\eta_1)\boxtimes_{\delta}\p1=\{z\in \mathcal{R}(\eta_1)\boxtimes_{\delta}\p1,
\ {\rm Res}_{\zp}z\in{\cal R}^+(\eta_1),\ {\rm Res}_{\zp}w\cdot z\in{\cal R}^+(\eta_1)\}.$$
  
   \begin{prop}\label{split} If $\varepsilon^{-1}\eta_1\eta_2=\delta$ for some continuous characters 
   $\eta_1,\eta_2, \delta$, then there is an exact sequence of topological $G$-modules 
    $$0\to B^{\rm an}(\eta_2,\eta_1)^*\otimes \delta\to \mathcal{R}(\eta_1)\boxtimes_{\delta} \p1\to B^{\rm an}(\eta_1,\eta_2)\to 0$$
    and a $G$-equivariant 
isomorphism $B^{\rm an}(\eta_2,\eta_1)^*\otimes \delta\cong \mathcal{R}^+(\eta_1)\boxtimes_{\delta}\p1$
of topological vector spaces. 
   \end{prop} 
   
     From now on we suppose that $D=D(s)\in \mathcal{MF}(\delta^{-1})$ for some
$s=(\delta_1,\delta_2,\mathcal{L})\in\mathcal{S}_{\rm irr}$ and
     some unitary character $\delta:\qpet\to \mathcal{O}_L^*$, and we let  
        $\eta=\delta\delta_D^{-1}$. By proposition~\ref{compatible}
we have $\check{D}\in \mathcal{MF}(\delta)$ and we let $\check{\Pi}=\Pi_{\delta^{-1}}(\check{D})$. 
  Since $\dim D=2$, there is a natural isomorphism 
  $\check{D}\cong D\otimes\delta_D^{-1}$.
  Combining these observations with part (i) of proposition~\ref{compatible} and with corollary VI.12 of \cite{CD}, 
   we obtain the following result. 
   
   \begin{lem} \label{exact}
We have $D\in \mathcal{MF}(\eta\delta_D^{-1})$ and there is a natural isomorphism
$\check{\Pi}\cong \Pi_{\delta_D \eta^{-1}}(D)\otimes \delta_D^{-1}$ as well as 
 an exact sequence of topological $G$-modules
    $$0\to (\check{\Pi}^{\rm an})^*\to D_{\rm rig}\boxtimes_{\delta}\p1\to \Pi^{\rm an}\to 0.$$
   \end{lem}

  \begin{lem}\label{devisser}
  There is an exact sequence of topological $G$-modules 
   $$0\to \mathcal{R}(\delta_1)\boxtimes_{\delta}\p1\to D_{\rm rig}\boxtimes_{\delta}\p1\to \mathcal{R}(\delta_2)\boxtimes_{\delta}\p1\to 0.$$
  \end{lem}
  
  \begin{proof}
  Since $\kappa(\delta)=\kappa(\delta_D)$ by remark~\ref{locconstant}, the desired result is proved in the same way as corollary 3.6 of \cite{Jacquet}. 
  For the reader's convenience, we sketch the argument. Let $a,b$ be the Hodge-Tate weights of $D$. 
Extension
 by $0$ allows to view $D_{\rm rig}=D_{\rm rig}\boxtimes_{\delta}\zp$ as a subspace
of $D_{\rm rig}\boxtimes_{\delta}\p1$ and any element of $D_{\rm rig}\boxtimes_{\delta}\p1$ can be
written as $z_1+w\cdot z_2$ with $z_1,z_2\in D_{\rm rig}$.
The equality 
$\kappa(\delta)=\kappa(\delta_D)$ combined with theorem 3.1 in \cite{ExtdR} yields 
$$u^{+}(z_1)=tz_1, \quad u^+(w\cdot z_2)=-w\cdot\frac{(\nabla-a)(\nabla-b)z_2}{t},\quad{\text{where $t=\log(1+T)$.}}$$
  We deduce that $X:=(D_{\rm rig}\boxtimes_{\delta}\p1)^{u^+=0}$ is isomorphic as an $L$-vector space to the space of solutions of the equation 
  $(\nabla-a)(\nabla-b)z=0$ on $D_{\rm rig}$. Proposition 2.1 of \cite{Jacquet} shows that $X$ is finite dimensional and lemma 2.6 of 
  loc.cit implies that all elements of $X$ are invariant under the action of the upper unipotent subgroup $U$ of $G$. 
  In particular, if $e_1$ is a basis of $\mathcal{R}(\delta_1)$ which is an eigenvector of $\varphi$ and $\Gamma$, then 
    $(0,e_1)\in X$ is $U$-invariant. Then the arguments in $\S~3.2$ of \cite{Jacquet} go through by replacing $\delta_D$ by 
    $\delta$. The result follows.
  \end{proof}

    \begin{lem}  \label{morphism}
  There is a morphism $B^{\rm an}(\delta_1,\eta\delta_2)\to \Pi^{\rm an}$ with a finite dimensional kernel.
   \end{lem}
  \begin{proof} 
Consider the inclusions $\mathcal{R}^+(\delta_1)\boxtimes_{\delta}\p1\subset D_{\rm rig}\boxtimes_{\delta}\p1$ given by lemma~\ref{devisser} and $(\check{\Pi}^{\rm an})^*\subset D_{\rm rig}\boxtimes_{\delta}\p1$ given by lemma~\ref{exact}. It follows from \cite[cor.~VI.14]{CD} 
     that $\mathcal{R}^+(\delta_1)\boxtimes_{\delta}\p1\subset (\check{\Pi}^{\rm an})^*$, hence there is a  morphism
     $$(\mathcal{R}(\delta_1)\boxtimes_{\delta}\p1)/ (\mathcal{R}^+(\delta_1)\boxtimes_{\delta}\p1)
\to (D_{\rm rig}\boxtimes_{\delta}\p1)/  (\check{\Pi}^{\rm an})^*.$$
     The left hand-side is isomorphic to $B^{\rm an}(\delta_1, \eta \delta_2)$ by proposition~\ref{split} and the right hand-side is isomorphic to $\Pi^{\rm an}$ by lemma~\ref{exact}. Consequently, we
    obtain a morphism
    $B^{\rm an}(\delta_1,\eta\delta_2)\to \Pi^{\rm an}$, whose kernel is a closed subspace of 
    $B^{\rm an}(\delta_1, \eta\delta_2)$, thus a space of compact type. On the other hand, the kernel is isomorphic to
    the quotient of $ (\check{\Pi}^{\rm an})^*\cap 
(\mathcal{R}(\delta_1)\boxtimes_{\delta}\p1)$ by the closed subspace 
    $\mathcal{R}^{+}(\delta_1)\boxtimes_{\delta}\p1$. Let $\sigma: D_{\rm rig}\to \mathcal{R}(\delta_2)$ be the natural projection.
    Then $$ (\check{\Pi}^{\rm an})^*\cap (\mathcal{R}(\delta_1)\boxtimes_{\delta}\p1)=\{z\in (\check{\Pi}^{\rm an})^*| \sigma({\rm Res}_{\zp}(z))=\sigma({\rm Res}_{\zp}(wz))=0\}$$
    is closed in the Fr\'echet space $ (\check{\Pi}^{\rm an})^*$, thus it is itself a Fr\'echet space. Since a Fr\'echet space which is also a space
    of compact type is finite dimensional, the result follows.     
    \end{proof}
  
  \begin{remark} Combining lemmas
\ref{exact} and~\ref{morphism}, we also obtain the existence of a morphism 
   $ B^{\rm an}(\delta_1, \eta^{-1}\delta_2)\to \Pi_{\delta_D \eta^{-1}}(D)^{\rm an}$ 
   with a finite dimensional kernel. 
  
  \end{remark}
  
  \subsubsection{Universal unitary completions and completion of the proof}
    The next theorem requires some preliminaries. 
   We say that 
   $s=(\delta_1,\delta_2,\mathcal{L})\in\mathcal{S}_{\rm irr}$ is 
   
   $\bullet$ exceptional if 
$\kappa(s)\in \mathbf{N}^*$
and $\delta_1=x^{\kappa(s)}\delta_2$
(in particular, $s\in\mathcal{S}_*^{\rm cris}$).
   
   $\bullet$ special if $\kappa(s)\in\mathbf{N}^*$ and 
   $\delta_1=x^{\kappa(s)-1}\varepsilon\delta_2$ 
(this includes $s\in\mathcal{S}_*^{\rm st}$). 
   
   If $s$ is special, then setting $W(\delta_1,\delta_2)={\rm Sym}^{\kappa(s)-1}(L^2)\otimes_{L} \delta_2$ there is~\cite[th. 2.7, rem. 2.11]{Cvectan} a natural isomorphism
   $${\rm Ext}^1_{G}(W(\delta_1,\delta_2), B^{\rm an}(\delta_1,\delta_2)/W(\delta_1,\delta_2))\cong {\rm Ext}^1(\mathcal{R}(\delta_2), \mathcal{R}(\delta_1)).$$
   The extension $D(s)_{\rm rig}$ of 
   $\mathcal{R}(\delta_2)$ by $\mathcal{R}(\delta_1)$ associated to $s$ gives therefore rise to an extension 
  $E_{\mathcal{L}}$ of 
   $W(\delta_1,\delta_2)$ by $ B^{\rm an}(\delta_1,\delta_2)/W(\delta_1,\delta_2)$ (these extensions were introduced and studied by Breuil \cite{L, brL}).
    
If $s=(\delta_1,\delta_2,\mathcal{L})\in\mathcal{S}_{\rm irr}$ we have $D(s)\in \mathcal{MF}(\delta_{D(s)}^{-1})$ by proposition~\ref{padicL}. We write 
   $\Pi(s)$ instead of $\Pi_{\delta_{D(s)}}(D(s))$.  Propositions~\ref{classical} and~\ref{compatible}(iv) imply that $\Pi(s)$ is in ${\rm Rep}_L(\delta_{D(s)})$ and is absolutely irreducible.

If $\pi$ is a representation of $G$ on a locally convex $L$-vector space, we let 
   $\widehat{\pi}$ be the universal unitary completion of $\pi$ (if it exists).

   \begin{thm}\label{universal}
  If $s=(\delta_1,\delta_2,\mathcal{L})\in\mathcal{S}_{\rm irr}$ then the following hold:
    
    a) If $s\in\mathcal{S}_{*}^{\rm cris}$ is not special, then\footnote{$B^{\rm alg}(\delta_1,\delta_2)$ is the
    space of ${\rm SL}_2(\qp)$ locally algebraic vectors in $B^{\rm an}(\delta_1,\delta_2)$.} $\Pi(s)=\widehat{B^{\rm alg}(\delta_1,\delta_2)}=\widehat{B^{\rm an}(\delta_1,\delta_2)}$.
        
    b) If $s\in\mathcal{S}_{*}^{\rm ng}$, then $\Pi(s)=\widehat{B^{\rm an}(\delta_1,\delta_2)}$.
    
    c) If $s$ is special (which includes the case $s\in\mathcal{S}_*^{\rm st}$), then $\Pi(s)=\widehat{E_{\mathcal{L}}}$.
  
      \end{thm}
   
   \begin{proof} Assume first that $s$ is not exceptional. Let $B(s)$ be the space of functions $\phi: \qp\to L$ of class $\mathcal{C}^{u(s)}$, 
   such that $x\mapsto (\delta_1\delta_2^{-1}\varepsilon^{-1})(x)\phi(1/x)$ 
extends to a function of class $\mathcal{C}^{u(s)}$. By results of Berger, Breuil and Emerton\footnote{See
   \cite[th. 4.3.1]{bb}, \cite[prop. 2.5]{pem}, \cite[cor.~3.2.3, 3.3.4]{L}.} one can express $\widehat{B^{\rm alg}(\delta_1,\delta_2)}$, $\widehat{B^{\rm an}(\delta_1,\delta_2)}$
   and $\widehat{E_{\mathcal{L}}}$ (according to whether $s\in \mathcal{S}_{*}^{\rm cris}, \mathcal{S}_{*}^{\rm ng}$ or $s$ special) as a quotient $\Pi_{\rm aut}(s)$ of $B(s)$. 
   Theorem IV.4.12 of \cite{Cbigone} (which builds on \cite{bb}, \cite{PS}, \cite{Ctrianguline}) shows that $D(s)\in \mathcal{MF}(\delta_{D(s)}^{-1})$ and that $\Pi_{\rm aut}(s)=\Pi(s)$, which finishes the proof in this case. 
   
 It remains to deal with the exceptional case\footnote{
     This problem is solved in \cite{except} for $p>2$.}.      
       Let 
  $\Pi=\Pi(s)$. The description of $\Pi(s)^{\rm an}$ given by \cite[prop. 4.11]{Cvectan} shows that 
  there is an injection $B^{\rm alg}(\delta_1,\delta_2)\to \Pi$. If
  $X=\widehat{B^{\rm alg}(\delta_1,\delta_2)}$, we obtain a morphism $X\to \Pi$ and an injection 
  $B^{\rm alg}(\delta_1,\delta_2)\to X$. In particular $X\ne 0$, and then
  the second paragraph of the proof of prop. $2.10$ in \cite{blocks}      
shows that we can find a non-exceptional point $s'\in\mathcal{S}_{*}^{\rm cris}$
  and lattices $\Theta_1,\Theta_2$ in $B^{\rm alg}(\delta_1,\delta_2)$ and $\Pi(s')^{\rm alg}$, both finitely generated as $\OO[G]$-modules and
  such that $\Theta_1/\varpi\cong \Theta_2/\varpi$.
  
    Since $\Theta_1,\Theta_2$ are finitely generated over $\OO[G]$, their $p$-adic completions are 
    open, bounded, $G$-stable lattices in $X$ and $\Pi(s')=\widehat{\Pi(s')^{\rm alg}}$, respectively. As $s'$ is not exceptional, we know (by the first paragraph) that $\Theta_2/\varpi$ is admissible, of finite length, thus $X$ is admissible, of finite length as Banach representation and $\overline{X}^{\rm ss}\cong \overline{\Pi(s')}^{\rm ss}$. In particular, the  
  image of the morphism $X\to \Pi$ is closed \cite{iw}. Since $\Pi$ is irreducible 
 we obtain an exact sequence $0\to K\to X\to \Pi\to 0$ in $\Ban^{\adm}_G(L)$. It follows from  \cite[lem.5.5]{comp} that this induces an exact
 sequence $0\to \overline{K}^{\rm ss}\to \overline{X}^{\rm ss}\to \overline{\Pi}^{\rm ss}\to 0$. Thus we have a surjection $\overline{\Pi(s')}^{\rm ss}\to \overline{\Pi(s)}^{\rm ss}$. Compatibility of 
  $p$-adic and mod $p$ local Langlands (\cite{Becomp} or proposition~\ref{compat}) implies that this surjection must be an isomorphism, which in turn shows that 
  $\overline{K}^{\rm ss}=0$, hence $K=0$. We conclude that $X\cong \Pi$ and we are done. 
    \end{proof}
   
\begin{prop}\label{Triangle}
  If $D\in \mathcal{S}_{\rm irr}$ and $D\in \mathcal{MF}(\delta^{-1})$ for some unitary character $\delta$, then $\delta=\delta_D$.
\end{prop}
 
\begin{proof}
    Write 
    $s=(\delta_1,\delta_2,\mathcal{L})\in\mathcal{S}_{\rm irr}$ and $\delta=\delta_D\eta$ (note that $\delta_D=\varepsilon^{-1}\delta_1\delta_2$).
    We will prove that $\eta=1$. 
      
     We start by proving that $\eta=\eta^{-1}$. Suppose that this is not the case and let 
    $s'=(\delta_1,\eta\delta_2,\mathcal{L})$ and $s''=(\delta_1,\eta^{-1}\delta_2,\mathcal{L})$. 
    Since $\eta$ is locally constant and unitary, we have 
  $s', s''\in\mathcal{S}_{\rm irr}$ and $s, s', s''$ are pairwise distinct. 
    At least one of $s', s''$ is not special, and replacing\footnote{This uses lemma~\ref{exact}.} $\eta$ by $\eta^{-1}$ we may assume that $s'$ has this property. 
    Lemma~\ref{morphism} gives a nonzero morphism $B^{\rm an}(\delta_1,\eta\delta_2)\to \Pi$. Applying 
theorem~\ref{universal} (to this morphism or to its restriction to $B^{\rm alg}(\delta_1,\eta\delta_2)$) yields a nonzero morphism 
  $\Pi(s')\to \Pi$. This must be an isomorphism since both the source and target are topologically irreducible and admissible by proposition~\ref{compatible}. Applying the functor $\Pi\mapsto{\bf V}(\Pi)$
and using proposition~\ref{compatible} again yields $D(s)\cong D(s')$, contradicting proposition~\ref{classical}. Thus $\eta=\eta^{-1}$, 
and the proof also shows that if $s'$ is not special, then $\eta=1$.
    
        Assume that $s'$ is special. Since $\eta^2=1$, we have $\Pi_{\delta_D\eta^{-1}}(D)=\Pi$ and the exact sequence in proposition~\ref{exact} becomes
    $$0\to (\Pi^{\rm an})^*\otimes\delta_D\to D_{\rm rig}\boxtimes_{\delta}\p1\to \Pi^{\rm an}\to 0.$$
    Proposition~\ref{split} also gives exact sequences 
       $$0\to B^{\rm an}(\delta_2, \eta\delta_1)^*\otimes \delta_D\to \mathcal{R}(\delta_1)\boxtimes_{\delta}\p1\to B^{\rm an}(\delta_1,\eta \delta_2)\to 0,$$
     $$  0\to B^{\rm an}(\delta_1, \eta \delta_2)^*\otimes\delta_D\to \mathcal{R}(\delta_2)\boxtimes_{\delta}\p1\to B^{\rm an}(\delta_2, \eta\delta_1)\to 0
     $$
    We are now exactly in the context of the proof of prop. 4.11, part ii) of \cite{Cvectan}, which shows that $\Pi^{\rm an}$ contains an extension 
    $E_{\mathcal{L}'}$ of $W(\delta_1,\eta\delta_2)$ by $B^{\rm an}(\delta_1,\eta\delta_2)/W(\delta_1,\eta\delta_2)$. This extension is necessarily non split since 
    $\Pi^{\rm an}$ does not contain any finite dimensional $G$-invariant subspace. 
     If $s''=(\delta_1,\eta\delta_2,\mathcal{L'})$, then the inclusion $E_{\mathcal{L}'}\to \Pi^{\rm an}$ induces via theorem~\ref{universal} a nonzero morphism $\Pi(s'')\to \Pi$. Arguing as in the previous paragraph, we obtain $D(s'')\cong D(s)$ and we conclude using proposition~\ref{classical}.  
\end{proof}   
      
\subsection{Representations of infinite height}\label{injheight}

\subsubsection{$(\varphi,\Gamma)$-modules of infinite height}
  
    In this $\S$ we fix a character $\delta: \qpet\to \mathcal{O}^{\times}$ and an absolutely irreducible 
   $D\in \mathcal{MF}(\delta^{-1})$ such that $D^+=\{0\}$. Let $\Pi=\Pi_{\delta}(D)$ and $\check{\Pi}=\Pi_{\delta^{-1}}(\check{D})$. By proposition~\ref{compatible} 
we have an inclusion $\check{\Pi}^*\subset D\boxtimes_{\delta}\p1$. We will use several times the inclusion $D^{\psi=\alpha}\subset {\rm Res}_{\zp} (\check{\Pi}^*)$ for all $\alpha\in \mathcal{O}^{\times}$, see the discussion in remark V.14 of \cite{CD}. 
Recall that $\mathcal{C}^{\alpha}=(1-\alpha\varphi)D^{\psi=\alpha}$. 
                
           \begin{prop}\label{inj}
          {\rm a)}  ${\rm Res}_{a+p^n\zp}: \check{\Pi}^*\to D$ is injective for $a\in\zp$ and $n\geq 0$. 
                    
          {\rm b)} $\mathcal{C}^{\alpha}\cap \mathcal{C}^{\beta}=\{0\}$ for all distinct $\alpha, \beta\in \mathcal{O}^{\times}$.
          
           \end{prop}
         
         \begin{proof}
       a)  ${\rm Res}_{a+p^n\zp}(z)=0$ is equivalent to ${\rm Res}_{\zp}(\matrice{p^{-n}}{0}{0}{1} \matrice{1}{-a}{0}{1} z)=0$, so it suffices to prove that ${\rm Res}_{\zp}: \check{\Pi}^*\to D$ is injective. Let $D_0$ be a stable lattice in $D$ and let 
       $\Pi_0=\Pi_{\delta}(D_0)$ and $\check{\Pi}_0=\Pi_{\delta^{-1}}(\check{D}_0)$. Then $\Pi_0$ and $\check{\Pi}_0$ are open, bounded and $G$-invariant lattices
       in $\Pi$ and $\check{\Pi}$. 
           Suppose that $z\in \check{\Pi}^*$ satisfies ${\rm Res}_{\zp}(z)=0$. Multiplying $z$
       by a power of $p$, we may assume that $z\in \check{\Pi}_0^*$. If $w=\matrice{0}{1}{1}{0}$, then $\matrice{p^n} {0}{0} {1}wz\in D_0\cap \check{\Pi}_0^*$ for all $n\geq 1$, thus $\varphi^n(wz)= {\rm Res}_{\zp} (\matrice{p^n}{0}{0}{1}wz)\in {\rm Res}_{\zp}(\check{\Pi}_0^{*})$. Since ${\rm Res}_{\zp}(\check{\Pi}_0^*)$ is compact (because 
       $\check{\Pi}_0^*$ is compact and ${\rm Res}_{\zp}$ is continuous), we deduce that $wz\in D_0^+=\{0\}$ and so $z=0$.

       b)  Let $x\in D^{\psi=\alpha}$ and $y\in D^{\psi=\beta}$ be such that $(1-\alpha\varphi)x=(1-\beta\varphi)y$. 
       Then $x-y=\varphi(\alpha x-\beta y)$, so ${\rm Res}_{1+p\zp}(x-y)=0$. Since $D^{\psi=\alpha}, D^{\psi=\beta}\subset {\rm Res}_{\zp} (\check{\Pi}^*)$, 
        we can write
       $x-y={\rm Res}_{\zp}(z)$ for some $z\in \check{\Pi}^*$. Then ${\rm Res}_{1+p\zp}(z)=0$, and
       part a) shows that $z=0$, thus $x=y$. But then $\alpha x=\beta y$ and so $x=y=0$. The result follows. 
                      \end{proof}

   \begin{cor}\label{trivial}
      Let $\eta:\qpet\to \mathcal{O}^{\times}$ be a locally constant character and let $\alpha,\beta\in \mathcal{O}^{\times}$. If 
      $m_{\eta}(\mathcal{C}^{\alpha})\cap \mathcal{C}^{\beta}\ne \{0\}$, then $\eta|_{\zpet}=1$ and $\alpha=\beta$.
     
    \end{cor}
    
    \begin{proof}  Suppose that $z\in \mathcal{C}^{\alpha}$ and $y\in\mathcal{C}^{\beta}$ are nonzero and satisfy 
    $m_{\eta}(z)=y$. Choose 
$\tilde{z}, \tilde{y}\in \check{\Pi}^*$ such that 
$z={\rm Res}_{\zpet}(\tilde{z})$ and $y={\rm Res}_{\zpet}(\tilde{y})$ (this uses the fact that $D^{\psi=\alpha}, D^{\psi=\beta}\subset {\rm Res}_{\zp} (\check{\Pi}^*)$). The hypothesis and 
part e) of proposition~\ref{mult} yield the existence of $n\geq 1$ such that 
$$m_{\eta}=\sum_{i\in (\mathbb{Z}/p^n\mathbb{Z})^{\times}} \eta(i) {\rm Res}_{i+p^n\zp}.$$
   
   For $i\in (\mathbb{Z}/p^n\mathbb{Z})^{\times}$, 
    applying ${\rm Res}_{i+p^n\zp}$ to 
the equality $m_{\eta}(z)=y$ (and using part d) of proposition~\ref{mult}) gives
$$\eta(i) {\rm Res}_{i+p^n\zp}(\tilde{z})={\rm Res}_{i+p^n\zp}(\tilde{y}),$$
hence (proposition~\ref{inj}) $\eta(i)\tilde{z}=\tilde{y}$. Since this holds for all 
$i\in (\mathbb{Z}/p^n\mathbb{Z})^{\times}$ and since $\tilde{z}\ne 0$, we infer that $\eta|_{\zpet}=1$. 
But then $m_{\eta}$ is the identity map and so the hypothesis becomes $\mathcal{C}^{\alpha}\cap \mathcal{C}^{\beta}\ne \{0\}$. 
Proposition~\ref{inj} shows that $\alpha=\beta$ and finishes the proof. 
 \end{proof}

  \subsubsection{A family of unramified twists of $D$}\label{familiar}
  
     In this $\S$ we let $V$ be any absolutely irreducible $L$-representation of $\cal{G}_{\qp}$
     of dimension $\geq 2$ and we let $V_0$ be a $\cal{G}_{\qp}$-stable $\mathcal{O}$-lattice in $V$. 
      Let $S=\mathcal{O}[[X]]$ and let $\delta^{\rm nr}: \cal{G}_{\qp}\to S^{\times}$ be the unramified character
   sending a geometric Frobenius to $1+X$. Then $V_{0,\rm un}=S\otimes_{\mathcal{O}} V_0$ 
    becomes a $\cal{G}_{\qp}$ module for the diagonal action. 
    
      Let $D_0$ (resp. $D_{0,\rm un}$) be the \'etale $(\varphi,\Gamma)$-module associated to  
       $V_0$ (resp. $V_{0,\rm un}$) by Fontaine's \cite{FoGrot} equivalence of categories and its version for families \cite{Dee}. 
       Concretely, $D_{0,\rm un}=\cal{O}_{\mathcal{E},S}\otimes_{\oe} D_0$, where 
       \footnote{The limit is taken for the $\mathfrak{m}_S=(\varpi,X)$-adic topology.}
          $$\cal{O}_{\mathcal{E},S}=S\widehat{\otimes}_{\cal O} \oe=\{\sum_{n\in\mathbb{Z}} a_n T^n, a_n\in S, \lim_{n\to-\infty} a_n=0\},$$
     $\gamma\in\Gamma$ and $\varphi$ acting by $\gamma\otimes\gamma$ and
     $\varphi(\lambda\otimes z)=((1+X)\varphi(\lambda))\otimes \varphi(z)$. 
    
    For $\alpha \in 1+{\mathfrak m}_L$, there is a surjective specialization map ${\rm sp}_{\alpha}: S\to \mathcal{O}$, sending $X$ to $\alpha-1$, with kernel 
     $\wp_{\alpha}=(X-\alpha+1)$. The induced map 
       $${\rm sp}_{\alpha}: \cal{O}_{\mathcal{E}, S}\to \oe, \quad s_{\alpha} (\sum_{n\in\mathbb{Z}} a_n T^n)=\sum_{n\in\mathbb{Z}} a_n(\alpha-1) T^n$$
       gives rise to a specialization map ${\rm sp}_{\alpha}: D_{0,\rm un}\to D_0$, which in turn 
       induces an isomorphism of $(\varphi,\Gamma)$-modules $D_{0,\rm un}/\wp_{\alpha}\cong D_0\otimes \alpha^{v_p}$. In particular, 
 ${\rm sp}_{\alpha}: D_{0,\rm un}\to D_0$ induces a $\Gamma$-equivariant morphism ${\rm sp}_{\alpha}: D_{0,\rm un}^{\psi=1}\to D_0^{\psi=\alpha}$.
Let $D_{\rm un}=L\otimes_{\mathcal{O}} D_{0,\rm un}$ and $D=L\otimes_{\cal O} D_0$.

       \begin{prop}\label{surj}
      For all $\alpha\in 1+{\mathfrak m}_L$ the map ${\rm sp}_{\alpha}: D_{\rm un}^{\psi=1}\to D^{\psi=\alpha}$
        is surjective.
        
      \end{prop}
       
              \begin{proof} Let $D_n=D_{0,\rm un}/\wp_{\alpha}^n$. 
               It suffices to prove that the cokernel of the natural map $D_{0,\rm un}^{\psi=1}\to D_{1}^{\psi=1}$ is 
             $\mathcal{O}$-torsion. The snake lemma applied to the sequence $0\to D_{n-1}\to D_n\to D_1\to 0$ mapped to itself by 
             $\psi-1$
             yields an exact sequence of $\mathcal{O}$-modules 
             $$0\to D_{n-1}^{\psi=1}\to D_n^{\psi=1}\to D_1^{\psi=1}\to \frac{D_{n-1}}{\psi-1}\to \frac{D_n}{\psi-1}\to \frac{D_1}{\psi-1}\to 0.$$
            All modules appearing in the exact sequence are compact \cite[prop. II.5.5, II.5.6]{Cmirab} and we have a natural isomorphism 
            $\varprojlim_{n} D_n^{\psi=1}=D_{0,\rm un}^{\psi=1}$ (as $D_{0,\rm un}=\varprojlim_{n} D_n$). Passing to the limit we obtain therefore an
            exact sequence 
            $$0\to D_{0,\rm un}^{\psi=1}\to D_{0,\rm un}^{\psi=1}\to D_1^{\psi=1}\to M\to M\to \frac{D_1}{\psi-1}\to 0,$$
            where $M=\varprojlim_{n} \frac{D_n}{\psi-1}$.  It is thus enough to prove that $M$ is a torsion $\mathcal{O}$-module.
            
           Let $\check{W}_n$ be the Galois representation associated to $\check{D}_n$, namely the Cartier dual of $W_n:=V_{0,\rm un}/\wp_{\alpha}^n$. It follows from \cite[remarque II.5.10]{Cmirab} that there is an isomorphism\footnote{Here $H={\rm Ker}(\varepsilon)={\rm Gal}(\overline{\qp}/\qp(\mu_{p^{\infty}}))$ and $X^{\vee}$ is the Pontryagin dual of $X$.}
$$\frac{D_n}{\psi-1}\cong [(\qp/\zp\otimes_{\zp} \check{D}_n)^{\varphi=1}]^{\vee}=[(\qp/\zp\otimes_{\zp} \check{W}_n)^{H}]^{\vee},$$
hence it suffices to check that $(\qp/\zp\otimes_{\zp} \check{W}_n)^{H}$ are $\mathcal{O}$-torsion modules of bounded exponent (as $n$ varies). 
            
Let $\mathcal{H}={\rm Gal}(\overline{\qp}/\mathbb{Q}_p^{\rm ab})$. 
Since 
$\check{V}_0$ is absolutely irreducible of dimension $\geq 2$, there is $N\geq 1$ such that $p^N$ kills $(\qp/\zp\otimes_{\zp} \check{V}_0)^{\mathcal{H}}$.
As $V_{0,\rm un}\cong S\otimes_{\mathcal{O}} V_0$, with $\mathcal{H}$ acting trivially on $S$, and since 
$S/\wp_{\alpha}^n$ is a finite free $\mathcal{O}$-module, we have 
$$(\qp/\zp\otimes_{\zp} \check{W}_n)^{H}\subset (\qp/\zp\otimes_{\zp} \check{W}_n)^{\mathcal{H}}=(S/\wp_{\alpha}^n\otimes_{\cal O} (L/\mathcal{O}\otimes_{\mathcal{O}} \check{V}_0))^{\mathcal{H}}$$
$$=S/\wp_{\alpha}^n\otimes_{\mathcal{O}} (L/\mathcal{O}\otimes_{\mathcal{O}} \check{V}_0)^{\mathcal{H}}$$ and the last 
module is killed by $p^N$. The result follows.               
                 \end{proof}

    \subsubsection{Analytic variation in the universal family}
 Recall that for $\alpha\in \mathcal{O}^{\times}$ we denote $\mathcal{C}^{\alpha}=(1-\alpha\varphi) D^{\psi=\alpha}$. 
We recall that there is \cite[prop. I.2.3]{Cmirab} a $\varphi$ and $\Gamma$-invariant perfect pairing $\{\,\,,\,\}: \check{D}\times D\to L$, under which $\varphi$ and $\psi$ are adjoints. 
The
following 
result follows from the proof of \cite[lemme VI.1.1]{Cmirab}.

        \begin{lem}\label{calpha}
         $\mathcal{C}^{\alpha}$ is the orthogonal (for the pairing $\{\,\,,\,\}$) of $\check{D}^{\psi=1/\alpha}$ inside $D^{\psi=0}$.
        \end{lem}
        
              Let $q=p$ if $p>2$ and $q=4$ if $p=2$. Fix 
      a topological generator $\gamma$ of 
   $1+q\zp$ and define a map $\ell:\zpet\to \zp$ by
      $$\ell: \zpet\cong \mu(\qp)\times (1+q\zp)\to 1+q\zp=\gamma^{\zp}\cong \zp,$$
the second map being the natural projection and the last map sending $\gamma^x$ to $x$. 

\begin{lem}\label{operators}
For all 
  $\eta\in \hat{\mathcal{T}}^0(L)$ there is an equality of operators on $D\boxtimes\zpet$
 $$m_{\eta}=\sum_{n\geq 0} (\eta(\gamma)-1)^n m_{\binom{\ell}{n}}$$
 and $m_{\binom{\ell}{n}}(D_0\boxtimes\zpet)\subset D_0\boxtimes\zpet$.
 \end{lem}
 
 \begin{proof}
 For all 
  $\eta\in \hat{\mathcal{T}}^0(L)$ and $x\in\zpet$ we have 
  $$\sum_{n\geq 0} (\eta(\gamma)-1)^n \binom{\ell(x)}{n}=\eta(\gamma)^{\ell(x)}=\eta(\gamma^{\ell(x)})=\eta(x),$$
  the last equality being a consequence of the fact that $x^{-1}\cdot \gamma^{\ell(x)}\in\mu(\qp)$. 
    Hence 
     $$\eta|_{\zpet}=\sum_{n\geq 0} (\eta(\gamma)-1)^n \binom{\ell}{n},$$
   the series being uniformly convergent on $\zpet$. This yields the first part. The second part is a consequence of the fact that $\binom{\ell}{n}\in \zp$.
    \end{proof}
        
         We are now ready to prove a key technical ingredient in the proof of theorem~\ref{main1}. 
         We identify 
        $\hat{\mathcal{T}}^0(L)$ and 
       $(1+{\mathfrak m}_L)\times (1+{\mathfrak m}_L)$ via the map $\eta\mapsto (\eta(\gamma),\eta(p))$. 
       
       \begin{defi}\label{an} A subset 
       $S$ of $(1+{\mathfrak m}_L)\times (1+{\mathfrak m}_L)$ is called {\it Zariski closed} if it is defined by a system of equations of the form 
       $f(x-1,y-1)=0$, with 
      $f\in \cal O[[X,Y]]$.  
       
       \end{defi}

       \begin{prop}\label{subgroup}
      The set 
      $$H=\{\eta\in \hat{\mathcal{T}}^0(L)|\quad m_{\eta}(\mathcal{C}^{\alpha})=\mathcal{C}^{\alpha\eta(p)} \quad \forall \alpha\in 1+{\mathfrak m}_L\}$$
      is a Zariski closed subgroup of $\hat{\mathcal{T}}^0(L)$.
     \end{prop}

         \begin{proof} Since $\hat{\mathcal{T}}^0(L)\to {\rm Aut}_{L}(D^{\psi=0})$, $\eta\mapsto m_{\eta}$
    is a morphism of groups, 
     $H$ is a subgroup of $\hat{\mathcal{T}}^0(L)$.
    To conclude, it suffices to check that 
      $$H_{\alpha}=\{\eta\in\hat{\mathcal{T}}^0(L)| m_{\eta}(\mathcal{C}^{\alpha})\subset \mathcal{C}^{\alpha\eta(p)}\}$$
        is Zariski closed for all $\alpha\in 1+{\mathfrak m}_L$.
        
        Let us fix $\alpha\in 1+{\mathfrak m}_L$ and denote               
        $\mathcal{C}_{\rm un}=(1-\varphi)D_{0,\rm un}^{\psi=1}$ and $\mathcal{\check{C}}_{\rm un}=(1-\varphi)\check{D}_{0,\rm un}^{\psi=1}$, 
        where $D_{0, \rm un}$ and $\check{D}_{0,\rm un}$ were defined in 
        ${\rm n}^{\rm o}$~\ref{familiar}. 
        If $\eta\in\hat{\mathcal{T}}^0(L)$, it follows from proposition~\ref{surj} that the specialization maps 
        induce surjections $\mathcal{C}_{\rm un}\otimes_{\mathcal{O}} L\to \mathcal{C}^{\alpha}$ and $\mathcal{\check{C}}_{\rm un}\otimes_{\mathcal{O}} L\to \mathcal{\check{C}}^{1/\alpha\eta(p)}$.
        Since $\mathcal{C}^{\alpha\eta(p)}$ is the orthogonal of $\mathcal{\check{C}}^{1/\alpha\eta(p)}$ in $D^{\psi=0}$ (lemma~\ref{calpha}), it follows that 
                $$H_{\alpha}=\{\eta\in\hat{\mathcal{T}}^0(L)|\quad \{{\rm sp}_{1/\alpha\eta(p)}(\check{z}), m_{\eta}({\rm sp}_{\alpha}(z))\}=0 \quad \forall \check{z}\in\mathcal{\check{C}}_{\rm un}, z\in \mathcal{C}_{\rm un}\}.$$               
               
                 Fix $\check{z}\in\mathcal{\check{C}}_{\rm un}$ and $z\in \mathcal{C}_{\rm un}$. We can write 
                 $\check{z}=\sum_{k\geq 0} X^k \check{z}_k$ with 
                 $\check{z}_k\in \check{D}_0$. By definition 
                 $${\rm sp}_{1/\alpha\eta(p)}(\check{z})=\sum_{k\geq 0} (\frac{1}{\alpha\eta(p)}-1)^k \check{z}_k.$$
                 
    Combining this relation and lemma~\ref{operators}, we obtain
   $$\{{\rm sp}_{1/\alpha\eta(p)}(\check{z}), m_{\eta}({\rm sp}_{\alpha}(z))\}=\sum_{k,n\geq 0} (\frac{1}{\alpha\eta(p)}-1)^k (\eta(\gamma)-1)^n \{\check{z}_k, m_{\binom{\ell}{n}}({\rm sp}_{\alpha}(z))\}$$
   and the last expression is the evaluation at $(\eta(\gamma)-1, \eta(p)-1)$ of an element of $\cal{O}[[X,Y]]$. Thus $H_{\alpha}$ is a Zariski closed subset of 
   $(1+{\mathfrak m}_L)^2$, which finishes the proof of proposition~\ref{subgroup}. 
           \end{proof}

\subsubsection{The Zariski closure of $(a^n,b^n)_{n\geq 1}$}

  We refer the reader to definition~\ref{an} for the notion of Zariski closed subset of $(1+{\mathfrak m}_L)\times (1+{\mathfrak m}_L)$.

 \begin{prop}\label{Zariski}
  Let $a,b\in 1+{\mathfrak m}_L$. The Zariski closure of $\{(a^n, b^n)| n\geq 1\}$ is 
  
  $\bullet$ A finite subgroup of $\mu_{p^{\infty}}\times \mu_{p^{\infty}}$ if $\log a=\log b=0$.
  
  $\bullet$ The set $\{(x,x^s)| x\in 1+{\mathfrak m}_L\}$ (respectively $\{(x^s,x)| x\in 1+{\mathfrak m}_L\}$) if $\log b=s\log a$ (respectively
  $\log a=s\log b$), $s\in\zp$ and $(\log a,\log b)\ne (0,0)$.
  
  $\bullet$ $(1+{\mathfrak m}_L)\times (1+{\mathfrak m}_L)$ if $\log a$ and $\log b$ are linearly independent over $\qp$.
 
 \end{prop}
 \begin{proof}
  The first two cases are immediate, so 
assume that $\log a$ and $\log b$ are linearly independent over $\qp$. Suppose that $f\in \mathcal{O}[[X,Y]]$
    satisfies $f(a^n-1,b^n-1)=0$ for all $n\geq 1$. We will prove that $f=0$. We may assume that $b=a^s$, with $s\in \mathcal{O}-\zp$.  
   For $n\geq 1$ we have $$v_p\left(\binom{s}{n}\right)\geq -v_p(n!)>-\frac{n}{p-1},$$
  hence $x^s=\sum_{n\geq 0} \binom{s}{n}(x-1)^n$ is well-defined for $v_p(x-1)>\frac{1}{p-1}$ and 
    $x\mapsto x^s$ is analytic in this ball, with values in $1+{\mathfrak m}_L$.
Thus  $x\mapsto f(x, (1+x)^s-1)$
 is analytic on the ball $v_p(x)>\frac{1}{p-1}$ and vanishes at $a^n-1$ for all $n\geq 1$. Since $\log a\ne 0$, it follows that 
 $f(x, (1+x)^s-1)=0$ for all $v_p(x)>\frac{1}{p-1}$, consequently $f(T, (1+T)^s-1)=0$ in $L[[T]]$.    
Proposition~\ref{algebraic} below combined with the formal Weierstrass preparation theorem yield $f=0$, which is the desired result.
 \end{proof}

\begin{prop}\label{algebraic}
If $s\in \mathcal{O}-\zp$, then $(1+T)^s$ is transcendental over ${\rm Frac}(\mathcal{E}^+)$. 
 \end{prop}
\begin{proof} Denote $f=(1+T)^s$ and 
assume that $f$ is algebraic over $K:={\rm Frac}(\mathcal{E}^+)$. 
We need the following elementary result \cite[prop. 7.3]{DGS}. 

\begin{lem}
  $(1+T)^s\in\mathcal{R}^+$ if and only if $s\in\zp$.
\end{lem}

   We start with the case $f\in K$. Since a nonzero element of $\mathcal{E}^+$ 
   generates the same ideal in $\mathcal{E}^+$ as a nonzero polynomial, 
    we have 
   $K\subset \mathcal{R}$, thus $f\in \mathcal{R}\cap L[[T]]=\mathcal{R}^+$ and we are done
   by the previous lemma. 

  Next, assume that $f$ is algebraic over $K$ and $f\notin K$. Let 
  $P=X^n+a_{n-1}X^{n-1}+...+a_0\in K[X]$ be its minimal polynomial 
over $K$, with $n>1$. Consider the differential operator $\partial=(1+T)\frac{d}{dT}$ 
and observe that $\partial f=sf$. The equality $\partial(P(f))-ns P(f)=0$ can also be written as 
$$\sum_{k=0}^{n-1} (\partial a_k+s(k-n)a_k)f^k=0.$$
By minimality of $n$ we deduce that $\partial a_k+s(k-n)a_k=0$ for all $k<n$, in particular $\partial(a_0\cdot (1+T)^{-sn})=0$,
hence $a_0=c\cdot (1+T)^{sn}$ for some $c\in L^*$. Thus
$(1+T)^{sn}\in K$ and by the previous paragraph this gives $sn\in \zp$, which combined with 
$s\in \mathcal{O}$ yields $s\in\zp$, a contradiction. The result follows.
  \end{proof}
  
   The following result follows immediately from proposition~\ref{Zariski}.

 \begin{cor}\label{torsion}
  If $\mu_p\subset L$, then any nontrivial Zariski closed subgroup of $\hat{\mathcal{T}}^0(L)$ contains a nontrivial character of finite order. 
 \end{cor}

\begin{remark}
 The conclusion of proposition~\ref{Zariski} fails if we work with unbounded analytic functions instead of elements of $L\otimes_{\mathcal{O}} \mathcal{O}[[X,Y]]$ when defining
 the Zariski closure: if $a, b\in 1+{\mathfrak m}_L$ satisfy $(\log a, \log b)\ne (0,0)$, then $\log b\cdot \log(1+X)-\log a\cdot \log(1+Y)$ vanishes at $(a^n-1,b^n-1)$ for all $n\geq 1$. 

\end{remark}

\end{document}